\documentclass[a4paper,oneside,11pt]{article}

\usepackage[a4paper]{geometry}
\usepackage{aeguill}                   % ... or a4paper or a5paper or ... 
\usepackage{graphicx}
\usepackage{amsmath,amsfonts,amssymb,amsthm}
\usepackage{pstricks}
\usepackage{epstopdf}
\usepackage[utf8]{inputenc} 
\usepackage[T1]{fontenc} 
 
\usepackage[english]{babel} 

\title{The Tits alternative for the automorphism group of a free product}
\author{Camille Horbez}
\usepackage[all]{xy}

\usepackage{bbm}
\begin{document}
\maketitle
%\tableofcontents
\newtheorem{de}{Definition} [section]
\newtheorem{theo}[de]{Theorem} 
\newtheorem{prop}[de]{Proposition}
\newtheorem{lemma}[de]{Lemma}
\newtheorem{cor}[de]{Corollary}
\newtheorem{propd}[de]{Proposition-Definition}

\theoremstyle{remark}
\newtheorem{rk}[de]{Remark}
\newtheorem{ex}[de]{Example}
\newtheorem{question}[de]{Question}

\normalsize

\addtolength\topmargin{-.5in}
\addtolength\textheight{1.in}
\addtolength\oddsidemargin{-.045\textwidth}
\addtolength\textwidth{.09\textwidth}

\begin{abstract}
Let $G=G_1\ast\dots\ast G_k\ast F$ be a countable group which splits as a free product, where all groups $G_i$ are freely indecomposable and not isomorphic to $\mathbb{Z}$, and $F$ is a finitely generated free group. If for all $i\in\{1,\dots,k\}$, both $G_i$ and its outer automorphism group $\text{Out}(G_i)$ satisfy the Tits alternative, then $\text{Out}(G)$ satisfies the Tits alternative. As an application, we prove that the Tits alternative holds for outer automorphism groups of right-angled Artin groups, and of torsion-free groups that are hyperbolic relative to a finite family of virtually polycyclic groups.
\end{abstract}

\section*{Introduction}

In his celebrated 1972 paper \cite{Tit72}, Tits proved that any subgroup of a finitely generated linear group over an arbitrary field is either virtually solvable, or contains a rank two free subgroup. This dichotomy has since been shown to hold for various classes of groups, such as hyperbolic groups (Gromov \cite{Gro87}), mapping class groups of compact surfaces (Ivanov \cite{Iva84}, McCarthy \cite{McC85}), outer automorphism groups $\text{Out}(F_N)$ of finitely generated free groups (Bestvina, Feighn and Handel \cite{BFH00,BFH05}), groups acting freely and properly on a CAT(0) cube complex (Sageev and Wise \cite{SW05}), the group of polynomial automorphisms of $\mathbb{C}^2$ (Lamy \cite{Lam01}), groups of bimeromorphic automorphisms of compact complex Kähler manifolds (Oguiso \cite{Ogu06}), groups of birational transformations of compact complex Kähler surfaces (Cantat \cite{Can11}).

For the four first classes of groups mentioned above, as well as in Oguiso's theorem, a slightly stronger result than Tits' actually holds, since virtually solvable subgroups can be shown to be finitely generated and virtually abelian, with a bound on the index of the abelian subgroup (see \cite{BLM83} for the mapping class group case, see \cite{Ali02, BFH04} for the $\text{Out}(F_N)$ case, see \cite{BH99} for the case of groups acting on a CAT(0) cube complex). 

\begin{de}
A group $G$ \emph{satisfies the Tits alternative} if every subgroup of $G$ (finitely generated or not) is either virtually solvable, or contains a rank two free subgroup. 
\end{de}

More generally, we will make the following definition. The classical Tits alternative corresponds to the case where $\mathcal{C}$ is the class of virtually solvable groups.

\begin{de}
Let $\mathcal{C}$ be a collection of groups. A group $G$ \emph{satisfies the Tits alternative relative to $\mathcal{C}$} if every subgroup of $G$ either belongs to $\mathcal{C}$, or contains a rank two free subgroup.
\end{de} 

It is often interesting to show stability results for the Tits alternative: when a group $G$ is built in some way out of simpler subgroups $G_i$, it is worth knowing that one can deduce the Tits alternative for $G$ from the Tits alternative for the $G_i$'s. The Tits alternative is known to be stable under some basic group-theoretic constructions, such as passing to subgroups or to finite index supergroups; it is also stable under extensions -- we insist that it is important here to allow for subgroups of $G$ that are not finitely generated in the definition of the Tits alternative. Antolín and Minasyan established stability results of the Tits alternative for graph products of groups \cite{AM13}.
\\
\\
\indent Our main result is about deducing the Tits alternative for the outer automorphism group of a free product of groups $G_i$, under the assumption that all groups $G_i$ and $\text{Out}(G_i)$ satisfy it. A celebrated theorem of Grushko \cite{Gru40} states that any finitely generated group $G$ splits as a free product of the form $$G=G_1\ast\dots\ast G_k\ast F,$$ where for all $i\in\{1,\dots,k\}$, the group $G_i$ is nontrivial, not isomorphic to $\mathbb{Z}$, and freely indecomposable, and $F$ is a finitely generated free group. This \emph{Grushko decomposition} is unique in the sense that both the number $k$ of indecomposable factors, and the rank of the free group $F$, are uniquely determined by $G$, and the conjugacy classes of the freely indecomposable factors are also uniquely determined, up to permutation. 

Our main result reduces the study of the Tits alternative of the outer automorphism group of any finitely generated group to that of its indecomposable pieces. It answers a question of Charney and Vogtmann, who were interested in the Tits alternative for outer automorphisms of right-angled Artin groups.

\begin{theo}\label{Tits-intro}
Let $G$ be a finitely generated group, and let $$G:=G_1\ast\dots\ast G_k\ast F$$ be the Grushko decomposition of $G$. Assume that for all $i\in\{1,\dots,k\}$, both $G_i$ and $\text{Out}(G_i)$ satisfy the Tits alternative.\\ Then $\text{Out}(G)$ satisfies the Tits alternative.
\end{theo}

Again, we insist on the fact that when we assume that the groups $G_i$ and $\text{Out}(G_i)$ satisfy the Tits alternative, it is important to consider all their subgroups (finitely generated or not) in the definition of the Tits alternative, even if we are only interested in establishing this alternative for finitely generated subgroups of $\text{Out}(G)$. 

Under the assumptions of Theorem \ref{Tits-intro}, since the Tits alternative is stable under extensions, the full automorphism group $\text{Aut}(G)$ also satisfies the Tits alternative. When $k=0$, we get a new, shorter proof of the Tits alternative for the outer automorphism group $\text{Out}(F_N)$ of a finitely generated free group, that was originally established by Bestvina, Feighn and Handel \cite{BFH00,BFH05}. In particular, this gives a new proof of the Tits alternative for the mapping class group of a compact surface with nonempty boundary.

More generally, if $\mathcal{C}$ is a collection of groups that is stable under isomorphisms, contains $\mathbb{Z}$, and is stable under passing to subgroups, to extensions, and to finite index supergroups, we show that $\text{Out}(G)$ satisfies the Tits alternative relative to $\mathcal{C}$, as soon as all $G_i$ and $\text{Out}(G_i)$ do, see Theorem \ref{Tits}. This applies for example to the class of virtually polycyclic groups. Bestvina, Feighn and Handel actually proved the Tits alternative for $\text{Out}(F_N)$ relative to the collection of all abelian groups \cite{BFH04}, which does not follow from our main result. More generally, it would be of interest to know whether the version of Theorem \ref{Tits-intro} relative to the class of abelian groups holds.
\\
\\
\indent Theorem \ref{Tits-intro} can be applied to prove the Tits alternative for outer automorphism groups of various interesting classes of groups. In \cite{CV11}, Charney and Vogtmann proved the Tits alternative for the outer automorphism group of a right-angled Artin group $A_{\Gamma}$ associated to a finite simplicial graph $\Gamma$, under a homogeneity assumption on $\Gamma$. As noticed in \cite[Section 7]{CV11}, Theorem \ref{Tits-intro} enables us to remove this assumption. This was Charney and Vogtmann's original motivation for asking the question about the Tits alternative for the outer automorphism group of a free product. Basically, when $\Gamma$ is disconnected, the group $A_{\Gamma}$ splits as a free product of the subgroups $A_{\Gamma_i}$ associated to its connected components, and Theorem \ref{Tits-intro} enables us to argue by induction on the number of vertices of $\Gamma$, using Charney and Vogtmann's results from \cite{CV11}.

\begin{theo} \label{Tits-raag}
For all finite simplicial graphs $\Gamma$, the group $\text{Out}(A_{\Gamma})$ satisfies the Tits alternative.
\end{theo}

Theorem \ref{Tits-intro} also applies to the outer automorphism group of a torsion-free group $G$ that is hyperbolic relative to a finite collection $\mathcal{P}$ of virtually polycyclic subgroups. Indeed, it enables to restrict to the case where $G$ is freely indecomposable relative to $\mathcal{P}$, i.e. $G$ does not split as a free product of the form $G=A\ast B$, where all subgroups in $\mathcal{P}$ are conjugate into either $A$ or $B$. In the freely indecomposable case, the group of outer automorphisms of $G$ was described by Guirardel and Levitt as being built out of mapping class groups and subgroups of linear groups \cite{GL14}.

\begin{theo} \label{Tits-relhyp}
Let $G$ be a torsion-free group that is hyperbolic relative to a finite collection of virtually polycyclic subgroups. Then $\text{Out}(G)$ satisfies the Tits alternative.
\end{theo}

More generally, if $G$ is a torsion-free group that is hyperbolic relative to a finite family of finitely generated parabolic subgroups, we show that if all parabolic subgroups, as well as their outer automorphism groups, satisfy the Tits alternative, then the subgroup of $\text{Out}(G)$ made of those automorphisms that preserve the conjugacy classes of all parabolic subgroups also satisfies the Tits alternative. We refer to Theorem \ref{tits-rh} for a precise statement.
\\
\\
\indent We now describe the main ideas in our proof of Theorem \ref{Tits-intro}. In the case of the mapping class group $\text{Mod}(S)$ of a compact surface $S$, one way of proving the Tits alternative is to start by proving the following trichotomy: every subgroup $H\subseteq\text{Mod}(S)$ either 

\begin{itemize}
\item contains two pseudo-Anosov diffeomorphisms of $S$ that generate a rank two free subgroup of $H$, or 
\item is virtually cyclic, virtually generated by a pseudo-Anosov diffeomorphism, or
\item virtually fixes the isotopy class of a simple closed curve on $S$. 
\end{itemize}

\noindent This trichotomy was proved by Ivanov in \cite{Iva92}, and independently by McCarthy and Papadopoulos in \cite{McP89}. They started by proving that every subgroup of $\text{Mod}(S)$ either contains a pseudo-Anosov diffeomorphism, or virtually fixes the isotopy class of a simple closed curve on $S$, before studying subgroups of $\text{Mod}(S)$ that contain a pseudo-Anosov diffeomorphism. Once the above trichotomy is established, a second step in the proof of the Tits alternative consists in arguing by induction, in the case where $H$ preserves the isotopy class of a simple closed curve $\gamma$. In this case, by cutting $S$ along $\gamma$, we get a collection of subsurfaces. The Tits alternative is proved by induction, by considering the restrictions of the diffeomorphisms in $H$ to these subsurfaces.
\\
\\
\indent Our proof of Theorem \ref{Tits-intro} follows the same strategy. For the inductive step, we will need to work with decompositions of $G$ into free products that are not necessarily equal to the Grushko decomposition. From now on, we let $G$ be a countable group that splits as a free product of the form $$G:=G_1\ast\dots\ast G_k\ast F,$$ where $F$ is a finitely generated free group, and all $G_i$ are nontrivial. We do not require this decomposition to be the Grushko decomposition of $G$: some factors $G_i$ can be equal to $\mathbb{Z}$, or be freely decomposable. We actually do not even require $G$ to be finitely generated: some $G_i$ might be infinitely generated (however the number $k$ of factors arising in the splitting is finite, and $F$ is finitely generated). We denote by $\mathcal{F}:=\{[G_1],\dots,[G_k]\}$ the finite set of all $G$-conjugacy classes of the $G_i$'s, which we call a \emph{free factor system} of $G$. We denote by $\text{Out}(G,\mathcal{F})$ the subgroup of $\text{Out}(G)$ made of those outer automorphisms of $G$ that send each $G_i$ to a conjugate. Theorem \ref{Tits-intro} is a particular case of the following version, which is suitable for our inductive arguments. 

\begin{theo}\label{theo-Tits-2}
Let $G$ be a countable group, and let $\mathcal{F}$ be a free factor system of $G$. Assume that for all $i\in\{1,\dots,k\}$, both $G_i$ and $\text{Out}(G_i)$ satisfy the Tits alternative relative to $\mathcal{C}$, where $\mathcal{C}$ is a collection of groups that is stable under isomorphisms, contains $\mathbb{Z}$, and is stable under subgroups, extensions, and passing to finite index supergroups.\\ Then $\text{Out}(G,\mathcal{F})$ satisfies the Tits alternative relative to $\mathcal{C}$. 
\end{theo}

As mentioned above, our proof of Theorem \ref{theo-Tits-2} will consist in two steps: establishing a trichotomy for subgroups $H\subseteq\text{Out}(G,\mathcal{F})$, and applying an inductive argument. The induction step consists in dealing with the case where $H$ virtually preserves the conjugacy class of a proper \emph{$(G,\mathcal{F})$-free factor}. A \emph{$(G,\mathcal{F})$-free factor} is a subgroup $A\subseteq G$ such that $G$ splits as a free product of the form $G=A\ast B$, and for all $i\in\{1,\dots,k\}$, the group $G_i$ is conjugate into either $A$ or $B$. A $(G,\mathcal{F})$-free factor is \emph{proper} if it is nontrivial, not conjugate to any of the $G_i$'s, and not equal to $G$. When $H$ preserves the conjugacy class of a proper free factor $A$,  the group $H$ is contained in $\text{Out}(G,\mathcal{F}')$, where $\mathcal{F'}$ is the free factor system of $G$ obtained from $\mathcal{F}$ by removing all subgroups in $\mathcal{F}$ that are conjugate into $A$, and replacing them by the $G$-conjugacy class of the factor $A$. When passing from $(G,\mathcal{F})$ to $(G,\mathcal{F}')$, some measure of complexity decreases, which enables us to argue by induction.  
\\
\\
\indent We now describe our analogue of Ivanov's trichotomy for subgroups of $\text{Out}(G,\mathcal{F})$. We first state an analogous trichotomy for subgroups of $\text{Out}(F_N)$. We recall that an automorphism $\Phi\in\text{Out}(F_N)$ is \emph{fully irreducible} if no nontrivial power of $\Phi$ preserves the conjugacy class of a proper free factor of $F_N$. Every subgroup of $\text{Out}(F_N)$ (finitely generated or not) either  

\begin{itemize}
\item contains two fully irreducible automorphisms that generate a rank two free subgroup, or 
\item is virtually cyclic, virtually generated by a fully irreducible automorphism, or 
\item virtually fixes the conjugacy class of a proper free factor of $F_N$. 
\end{itemize} 

In \cite{HM09}, Handel and Mosher proved that any finitely generated subgroup of $\text{Out}(F_N)$ either contains a fully irreducible automorphism, or virtually fixes the conjugacy class of a proper free factor. Their proof uses the same kinds of techniques as Bestvina, Feighn and Handel's proof of the Tits alternative \cite{BFH00}, so it cannot be used to get a new proof of the Tits alternative for $\text{Out}(F_N)$. The study of subgroups of $\text{Out}(F_N)$ that contain a fully irreducible element is due to Bestvina, Feighn and Handel \cite{BFH97}, another approach is due to Kapovich and Lustig \cite{KL11}. In \cite{Hor14-3}, we gave a new, shorter proof of the above trichotomy, independent from the work in \cite{BFH00}, that also works for non finitely generated subgroups of $\text{Out}(F_N)$.

Our proof of this statement uses the action of $\text{Out}(F_N)$ on the \emph{free factor complex} $FF_N$, whose hyperbolicity was originally proved by Bestvina and Feighn \cite{BF12}. Bestvina and Feighn also proved that an automorphism $\Phi\in\text{Out}(F_N)$ acts loxodromically on $FF_N$ if and only if $\Phi$ is fully irreducible. In terms of the action of $\text{Out}(F_N)$ on $FF_N$, the above trichotomy can be restated as follows: every subgroup of $\text{Out}(F_N)$ either

\begin{itemize}
\item contains a rank two free subgroup generated by two loxodromic isometries of $FF_N$, or
\item is virtually cyclic, virtually generated by a loxodromic isometry of $FF_N$, or
\item has a finite orbit in $FF_N$.
\end{itemize}

More generally, given a group $G$ acting by isometries on a (possibly non-proper) hyperbolic space $X$, it follows from a classification of groups of isometries of hyperbolic spaces due to Gromov \cite{Gro87} that either $G$ 

\begin{itemize}
\item contains a rank two free subgroup, generated by two loxodromic isometries of $X$, or
\item has a fixed point in the Gromov boundary $\partial_{\infty}X$, or
\item has a bounded orbit in $X$.
\end{itemize}

The key point for deducing the above trichotomy statement for subgroups of $\text{Out}(F_N)$ from Gromov's statement consists in showing that if $H$ has a bounded orbit in $FF_N$, then $H$ has a finite orbit in $FF_N$. This is not obvious because $FF_N$ is not locally finite. To bypass this difficulty, we studied stationary measures on the compact closure of Culler and Vogtmann's outer space $CV_N$, and projected them to the Gromov boundary of the complex of free factors. In our proof of the above trichotomy, we also need to understand stabilizers of points in $\partial_{\infty}FF_N$ for dealing with the second case in Gromov's theorem.
\\
\\
\indent We prove a similar trichotomy for subgroups of $\text{Out}(G,\mathcal{F})$, with $(G,\mathcal{F})$ as above. To this means, we work with relative outer space $P\mathcal{O}(G,\mathcal{F})$, and the complex of relative cyclic splittings $FZ(G,\mathcal{F})$. The geometry of these complexes was investigated in a series of previous papers \cite{Hor14-5,Hor14-6}. In \cite{Hor14-5}, we described a compactification $\overline{P\mathcal{O}(G,\mathcal{F})}$ of the relative outer space in terms of very small actions of $G$ on $\mathbb{R}$-trees. In \cite{Hor14-6}, we proved the hyperbolicity of the complex of relative cyclic splittings, and described its Gromov boundary as a quotient of a subspace $P\mathcal{X}(G,\mathcal{F})$ of $\overline{P\mathcal{O}(G,\mathcal{F})}$. Assume that the pair $(G,\mathcal{F})$ is \emph{nonsporadic}, i.e. we do not have $G=G_1\ast G_2$ and $\mathcal{F}=\{[G_1],[G_2]\}$, or $G=G_1\ast\mathbb{Z}$ and $\mathcal{F}=\{[G_1]\}$. The trichotomy that we prove for subgroups of $\text{Out}(G,\mathcal{F})$ is the following: every subgroup $H\subseteq\text{Out}(G,\mathcal{F})$ (finitely generated or not) either

\begin{itemize}
\item contains a rank two free subgroup, generated by two loxodromic isometries of $FZ(G,\mathcal{F})$, or
\item virtually fixes a tree with trivial arc stabilizers in $\partial P\mathcal{O}(G,\mathcal{F})$, or 
\item virtually preserves the conjugacy class of a proper $(G,\mathcal{F})$-free factor.
\end{itemize}

Again, the key point is to understand subgroups of $\text{Out}(G,\mathcal{F})$ with bounded orbits in $FZ(G,\mathcal{F})$. We show that if a subgroup $H\subseteq \text{Out}(G,\mathcal{F})$ does not virtually preserve the conjugacy class of any proper $(G,\mathcal{F})$-free factor, then the $H$-orbit of any point of $FZ(G,\mathcal{F})$ has a limit point in the Gromov boundary. 

Our argument relies on techniques coming from the theory of random walks on groups. Given a probability measure $\mu$ on $\text{Out}(F_N)$ whose support generates the subgroup $H$, we consider \emph{$\mu$-stationary} measures $\nu$ on $\overline{P\mathcal{O}(G,\mathcal{F})}$, i.e. probability measures that satisfy $$\nu(E)=\sum_{\Phi\in\text{Out}(G,\mathcal{F})}\mu(\Phi)\nu(\Phi^{-1}E)$$ for all $\nu$-measurable subsets $E\subseteq \overline{P\mathcal{O}(G,\mathcal{F})}$. Compactness of $\overline{P\mathcal{O}(G,\mathcal{F})}$ yields the existence of a $\mu$-stationary measure on $\overline{P\mathcal{O}(G,\mathcal{F})}$ that describes the repartition of accumulation points of sample paths of the random walk on $\text{Out}(G,\mathcal{F})$, realized on $P\mathcal{O}(G,\mathcal{F})$ via the action. This is the Markov chain whose position at time $n$ is obtained by successive multiplications on the right of $n$ independent automorphisms, all distributed with law $\mu$. We prove that any $\mu$-stationary measure $\nu$ on $\overline{P\mathcal{O}(G,\mathcal{F})}$ is supported on the subspace $P\mathcal{X}(G,\mathcal{F})$. The measure $\nu$ therefore projects to a $\mu$-stationary measure on the Gromov boundary of $FZ(G,\mathcal{F})$. The closure of the $H$-orbit of any point in $FZ(G,\mathcal{F})$ meets the support of $\nu$, which shows the existence of a limit point in the Gromov boundary. 

To prove the Tits alternative for $\text{Out}(G,\mathcal{F})$, we also need to understand subgroups of $\text{Out}(G,\mathcal{F})$ that stabilize a tree with trivial arc stabilizers in $\partial P\mathcal{O}(G,\mathcal{F})$, which is made possible by work of Guirardel and Levitt \cite{GL14-2}. When $H$ fixes the conjugacy class of a proper free factor, we argue by induction, as explained above. 
\\
\\
\indent As we are considering invariant free factors (and not invariant splittings) for the inductive step, it could seem to be more natural to work directly in the complex of proper $(G,\mathcal{F})$-free factors, whose hyperbolicity was recently proved by Handel and Mosher \cite{HM14-2}, and try to prove that every subgroup of $\text{Out}(G,\mathcal{F})$ either has a finite orbit, or has a limit point in the Gromov boundary. However, describing the Gromov boundary of the complex of proper $(G,\mathcal{F})$-free factors is still an open problem. We bypass this difficulty by working in the complex $FZ(G,\mathcal{F})$, whose Gromov boundary was described in \cite{Hor14-6}.
\\
\\
\noindent The paper is organized as follows. In Section \ref{sec-1}, we review basic facts about Gromov hyperbolic spaces, free products of groups, and relative spaces associated to them. In Section \ref{sec-2}, we deal with the \emph{sporadic} cases where either $G=G_1\ast G_2$ and $\mathcal{F}=\{[G_1],[G_2]\}$, or $G=G_1\ast\mathbb{Z}$ and $\mathcal{F}=\{[G_1]\}$. In Section \ref{sec-3}, we state Guirardel and Levitt's theorem about stabilizers of trees in $\overline{P\mathcal{O}(G,\mathcal{F})}$ that is needed in our proof of Theorem \ref{theo-Tits-2}. Section \ref{sec-4} contains a study of \emph{arational} $(G,\mathcal{F})$-trees, which is used in Section \ref{sec-5} to establish the trichotomy for subgroups of $\text{Out}(G,\mathcal{F})$. Theorem \ref{sec-6} is devoted to the inductive arguments. The reader will also find complete versions of our various statements of the Tits alternative in this section. Finally, in Section \ref{sec-7}, we give applications of our main result to automorphism groups of right-angled Artin groups, and of relatively hyperbolic groups.  

\section*{Acknowledgements}

It is a great pleasure to thank my advisor Vincent Guirardel for the many interesting discussions we had together. 
I acknowledge support from ANR-11-BS01-013 and from the Lebesgue Center of Mathematics.

\section{Review}\label{sec-1}

\subsection{Gromov hyperbolic spaces}

A geodesic metric space $(X,d)$ is \emph{Gromov hyperbolic} if there exists $\delta>0$ such that for all $x,y,z\in X$, and all geodesic segments $[x,y],[y,z]$ and $[x,z]$, we have $N_{\delta}([x,z])\subseteq N_{\delta}([x,y])\cup N_{\delta}([y,z])$ (where given $Y\subseteq X$, we denote by $N_{\delta}(Y)$ the $\delta$-neighborhood of $Y$ in $X$). The \emph{Gromov boundary} $\partial_{\infty} X$ of $X$ is the space of equivalence classes of quasi-geodesic rays in $X$, two rays being equivalent if their images lie at bounded Hausdorff distance (we recall that a \emph{quasi-geodesic ray} is a map $\gamma:\mathbb{R}_+\to X$, so that there exist $K,L>0$ such that for all $s,t\in\mathbb{R}_+$, we have $\frac{1}{K}|t-s|-L\le d(\gamma(s),\gamma(t))\le K|t-s|+L$). An isometry $\phi$ of $X$ is \emph{loxodromic} if for all $x\in X$, we have $$\lim_{n\to +\infty}\frac{1}{n}d(x,\phi^nx)>0.$$ Given a group $G$ acting by isometries on $X$, we denote by $\Lambda_XG$ the \emph{limit set} of $G$ in $\partial_{\infty} X$, which is defined as the intersection of $\partial_{\infty} X$ with the closure of the orbit of any point in $X$ under the $G$-action. The following theorem, essentially due to Gromov, gives a classification of isometry groups of (possibly nonproper) Gromov hyperbolic spaces. A sketch of proof can be found in \cite[Proposition 3.1]{CCMT13}, see also \cite[Theorem 2.7]{Ham13}.

\begin{theo} (Gromov \cite[Section 8.2]{Gro87})\label{Gromov-1}
Let $X$ be a hyperbolic geodesic metric space, and let $G$ be a group acting by isometries on $X$. Then $G$ is either 
\begin{itemize}
\item \emph{bounded}, i.e. all $G$-orbits in $X$ are bounded; in this case $\Lambda_X G=\emptyset$, or
\item \emph{horocyclic}, i.e. $G$ is not bounded and contains no loxodromic element; in this case $\Lambda_X G$ is reduced to one point, or
\item \emph{lineal}, i.e. $G$ contains a loxodromic element, and any two loxodromic elements have the same fixed points in $\partial_{\infty} X$; in this case $\Lambda_X G$ consists of these two points, or
\item \emph{focal}, i.e. $G$ is not lineal, contains a loxodromic element, and any two loxodromic elements have a common fixed point in $\partial_{\infty} X$; in this case $\Lambda_X G$ is uncountable and $G$ has a fixed point in $\Lambda_X G$, or
\item \emph{of general type}, i.e. $G$ contains two loxodromic elements with no common endpoints; in this case $\Lambda_X G$ is uncountable and $G$ has no finite orbit in $\partial_{\infty} X$. In addition, the group $G$ contains two loxodromic isometries that generate a rank two free subgroup. 
\end{itemize} 
\end{theo}

In particular, we have the following result.

\begin{theo} (Gromov \cite[Section 8.2]{Gro87})\label{Gromov}
Let $X$ be a hyperbolic geodesic metric space, and let $G$ be a group acting by isometries on $X$. If $\Lambda_X G\neq\emptyset$, and $G$ has no finite orbit in $\partial_{\infty} X$, then $G$ contains a rank two free subgroup generated by two loxodromic isometries.
\end{theo}

\subsection{Free factor systems and relative complexes} \label{sec-relative}

\paragraph*{Free factor systems.}

Let $G$ be a countable group that splits as a free product of the form $$G:=G_1\ast\dots\ast G_k\ast F,$$ where $F$ is a finitely generated free group. We let $\mathcal{F}:=\{[G_1],\dots,[G_k]\}$ be the finite collection of all $G$-conjugacy classes of the $G_i$'s. We fix a free basis $\{g_1,\dots,g_N\}$ of $F$, and we let $T^{\text{def}}$ be the $G$-tree defined as the Bass--Serre tree of the graph of group decomposition of $G$ depicted on Figure \ref{fig-def}. The rank of the free group $F$ arising in the splitting of $G$ only depends on $\mathcal{F}$. We call it the \emph{free rank} of $(G,\mathcal{F})$ and denote it by $\text{rk}_f(G,\mathcal{F})$. The \emph{Kurosh rank} of $(G,\mathcal{F})$ is defined as $\text{rk}_K(G,\mathcal{F}):=|\mathcal{F}|+\text{rk}_f(G,\mathcal{F})$. 

\begin{figure}
\begin{center}
\input{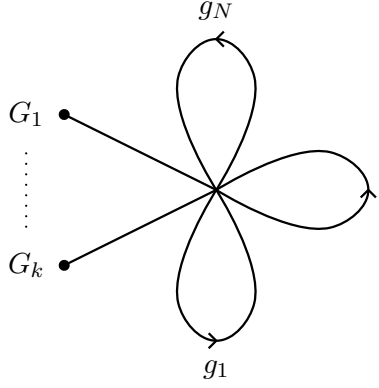}
\caption{The tree $T^{\text{def}}$ is the Bass--Serre tree of the above graph of groups decomposition of $G$.}
\label{fig-def}
\end{center}
\end{figure}

Subgroups of $G$ which are conjugate into one of the subgroups of $\mathcal{F}$ will be called \emph{peripheral} subgroups. A \emph{$(G,\mathcal{F})$-free splitting} is a minimal, simplicial $G$-tree $T$ in which all peripheral subgroups are \emph{elliptic} (i.e. they fix a point in $T$), and edge stabilizers are trivial. 

\paragraph*{Subgroups of free products.}

Subgroups of free products were studied by Kurosh in \cite{Kur34}. Let $H$ be a subgroup of $G$. By considering the $H$-minimal subtree in the tree $T^{\text{def}}$ (see the definition in Section \ref{sec-o} below), we get the existence of a (possibly infinite) set $J$, together with an integer $i_j\in\{1,\dots,k\}$, a nontrivial subgroup $H_{j}\subseteq G_{i_j}$ and an element $g_{j}\in G$ for each $j\in J$, and a (not necessarily finitely generated) free subgroup $F'\subseteq G$, so that $$H=\ast_{j\in J}~ g_{j}H_{j}g_{j}^{-1}\ast F'.$$ This splitting will be called the \emph{Kurosh decomposition} of $H$. The \emph{Kurosh rank} of $H$ is equal to $\text{rk}_K(H):=|J|+\text{rk}(F')$, its \emph{free rank} is $\text{rk}_f(H):=\text{rk}(F')$. They can be infinite in general. We also let $\mathcal{F}_H$ denote the set of $H$-conjugacy classes of the subgroups $g_{j}H_{j}g_{j}^{-1}$, which might also be infinite in general. We note that $\text{rk}_f(G,\mathcal{F})$ and $\mathcal{F}_H$ (and hence $\text{rk}_K(G,\mathcal{F})$) only depend on $H$ and $\mathcal{F}$, and not of our initial choice of $T^{\text{def}}$.

\paragraph*{Free factors.}

A \emph{$(G,\mathcal{F})$-free factor} is a subgroup of $G$ that is a point stabilizer in some $(G,\mathcal{F})$-free splitting. A $(G,\mathcal{F})$-free factor is \emph{proper} if it is nonperipheral (in particular nontrivial), and not equal to $G$. The Kurosh decomposition of a proper $(G,\mathcal{F})$-free factor reads as $$H=G'_{i_1}\ast\dots\ast G'_{i_r}\ast F',$$ where each of the subgroups $G'_{i_j}$ is conjugate in $G$ to one of the factors in $\mathcal{F}$ (with no repetition in the indices, i.e. the $G'_{i_j}$'s are pairwise non conjugate in $G$), and $F'$ is a finitely generated free group. In particular, the Kurosh rank of $H$ is finite. The group $G$ then splits as $$G=H\ast G'_{i_{r+1}}\ast\dots\ast G'_{i_k}\ast F'',$$ where $F''$ is a finitely generated free subgroup of $G$, and the $G'_{i_j}$'s are conjugate to the factors in $\mathcal{F}$ that do not arise in the Kurosh decomposition of $H$. The finite collection $\mathcal{F}':=\{[H],[G_{i_{r+1}}],\dots,[G_{i_k}]\}$ (where we consider $G$-conjugacy classes) is a free factor system of $G$, and we have
\begin{equation}\label{J}
|\mathcal{F}'|+|\mathcal{F}_H|=|\mathcal{F}|+1,
\end{equation}
\noindent and
\begin{equation}\label{rkf}
\text{rk}_f(G,\mathcal{F}')+\text{rk}_f(H)=\text{rk}_f(G,\mathcal{F}),
\end{equation}
\noindent whence
\begin{equation}\label{rkK}
\text{rk}_K(G,\mathcal{F}')+\text{rk}_K(H)=\text{rk}_K(G,\mathcal{F})+1.
\end{equation}

Let $H$ and $H'$ be two $(G,\mathcal{F})$-free factors, and let $T$ be a $(G,\mathcal{F})$-free splitting, one of whose elliptic subgroups is equal to $H$. By looking at the $H'$-minimal subtree of $T$, we see that $H\cap H'$ is an $(H',\mathcal{F}_{H'})$-free factor, so it is a $(G,\mathcal{F})$-free factor. This implies that the intersection of any family of $(G,\mathcal{F})$-free factors is again a free factor. In particular, any subgroup $A\subseteq G$ is contained in a smallest $(G,\mathcal{F})$-free factor, obtained as the intersection of all $(G,\mathcal{F})$-free factors that contain $A$. We denote it by $\text{Fill}(A)$.

\paragraph*{Relative automorphisms.}

Let $G$ be a countable group, and $\mathcal{F}$ be a free factor system of $G$. We denote by $\text{Out}(G,\mathcal{F})$ the subgroup of $\text{Out}(G)$ made of those automorphisms that preserve the conjugacy classes in $\mathcal{F}$. We denote by $\text{Out}(G,\mathcal{F}^{(t)})$ the subgroup of $\text{Out}(G)$ made of those automorphisms that act as a conjugation by an element of $G$ on each peripheral subgroup. 

For all $i\in\{1,\dots,k\}$, the group $G_i$ is equal to its normalizer in $G$. Therefore, any element of $\text{Out}(G)$ that preserves the conjugacy class of $G_i$ induces a well-defined outer automorphism of $G_i$. In other words, there is a morphism $$\text{Out}(G,\{[G_i]\})\to\text{Out}(G_i).$$ By taking the product over all groups $G_i$, we thus get a (surjective) morphism $$\text{Out}(G,\mathcal{F})\to\prod_{i=1}^k\text{Out}(G_i),$$ whose kernel is equal to $\text{Out}(G,\mathcal{F}^{(t)})$.

More generally, suppose that we are given a collection of subgroups $A_i\subseteq\text{Out}(G_i)$ for all $i\in\{1,\dots,k\}$, and let $\mathcal{A}=\{A_1,\dots,A_k\}$. We can define the subgroup $\text{Out}(G,\mathcal{F}^{\mathcal{A}})$ of $\text{Out}(G)$ made of those automorphisms that preserve all conjugacy classes in $\mathcal{F}$, and which induce an element of $A_i$ in restriction to $G_i$ for all $i\in\{1,\dots,k\}$. As above, there is a (surjective) morphism $$\text{Out}(G,\mathcal{F}^{\mathcal{A}})\to\prod_{i=1}^k A_i,$$ whose kernel is equal to $\text{Out}(G,\mathcal{F}^{(t)})$.

\subsection{Relative outer spaces}\label{sec-o}

An \emph{$\mathbb{R}$-tree} is a metric space $(T,d_T)$ in which any two points $x,y\in T$ are joined by a unique embedded topological arc, which is isometric to a segment of length $d_T(x,y)$. A $(G,\mathcal{F})$-tree is an $\mathbb{R}$-tree equipped with a minimal, isometric action of $G$, in which all peripheral subgroups of $G$ are \emph{elliptic}. We recall that an action on a tree is termed \emph{minimal} if there is no proper and nontrivial invariant subtree. Whenever a group $G$ acts on an $\mathbb{R}$-tree $T$, and some element of $G$ does not fix any point in $T$, there is a unique subtree of $T$ on which the $G$-action is minimal. In particular, whenever $H$ is a subgroup of $G$ that contains a hyperbolic element, we can consider the minimal subtree for the induced action of $H$ on $T$, which we call the \emph{$H$-minimal subtree} of $T$. The action of $H$ on $T$ is \emph{simplicial} if the $H$-minimal subtree is homeomorphic (when equipped with the topology defined by the metric) to a simplicial tree. We say that the action of $H$ on $T$ is \emph{relatively free} if all point stabilizers of the $H$-minimal subtree of $T$ are conjugate into $\mathcal{F}_H$.
 
A \emph{Grushko $(G,\mathcal{F})$-tree} is a simplicial $(G,\mathcal{F})$-tree with trivial edge stabilizers, all of whose elliptic subgroups are peripheral. Two $(G,\mathcal{F})$-trees are \emph{equivalent} if there exists a $G$-equivariant isometry between them.

The \emph{unprojectivized outer space} $\mathcal{O}(G,\mathcal{F})$, introduced by Guirardel and Levitt in \cite{GL07}, is defined to be the space of all equivalence classes of Grushko $(G,\mathcal{F})$-trees. \emph{Outer space} $P\mathcal{O}(G,\mathcal{F})$ is defined as the space of homothety classes of trees in $\mathcal{O}(G,\mathcal{F})$. Outer space, as well as its unprojectivized version, comes equipped with a right action of $\text{Out}(G,\mathcal{F})$, given by precomposing the actions (this can be turned into a left action by letting $\Phi.T:=T.\Phi^{-1}$ for all $T\in \mathcal{O}(G,\mathcal{F})$ and all $\Phi\in\text{Out}(G,\mathcal{F})$).

For all $g\in G$ and all $T\in \mathcal{O}(G,\mathcal{F})$, the \emph{translation length} of $g$ in $T$ is defined to be $$||g||_T:=\inf_{x\in T}d_T(x,gx).$$ Culler and Morgan have shown in \cite{CM87} that the map
\begin{displaymath}
\begin{array}{cccc}
i:&\mathcal{O}(G,\mathcal{F})&\to &\mathbb{R}^{G}\\
&T&\mapsto &(||g||_T)_{g\in G}
\end{array}
\end{displaymath}

\noindent is injective. We equip $\mathcal{O}(G,\mathcal{F})$ with the topology induced by this embedding, which is called the \emph{axes topology}. Outer space is then embedded as a subspace of the projective space $\mathbb{PR}^{G}$, and is equipped with the quotient topology. Its closure $\overline{P\mathcal{O}(G,\mathcal{F})}$, whose lift to $\mathbb{R}^{G}$ we denote by $\overline{\mathcal{O}(G,\mathcal{F})}$, is compact (see \cite[Theorem 4.2]{CM87} and \cite[Proposition 1.2]{Hor14-5}). We let $\partial P\mathcal{O}(G,\mathcal{F}):=\overline{P\mathcal{O}(G,\mathcal{F})}\smallsetminus P\mathcal{O}(G,\mathcal{F})$, and similarly $\partial\mathcal{O}(G,\mathcal{F}):=\overline{\mathcal{O}(G,\mathcal{F})}\smallsetminus\mathcal{O}(G,\mathcal{F})$. A $(G,\mathcal{F})$-tree $T$ is \emph{very small} if its arc stabilizers are either trivial, or maximally-cyclic and nonperipheral, and its tripod stabilizers are trivial. In \cite[Theorem 0.1]{Hor14-5}, we identified the space $\overline{P\mathcal{O}(G,\mathcal{F})}$ with the space of very small, minimal, projective $(G,\mathcal{F})$-trees. We also proved that it has finite topological dimension equal to $3\text{rk}_f(G,\mathcal{F})+2|\mathcal{F}|-4$.

\subsection{The cyclic splitting graph} \label{sec-hyp-fz}

Let $G$ be a countable group, and let $\mathcal{F}$ be a free factor system of $G$. A \emph{$\mathcal{Z}$-splitting} of $(G,\mathcal{F})$ is a minimal, simplicial $(G,\mathcal{F})$-tree, all of whose edge stabilizers are either trivial, or cyclic and nonperipheral. It is a \emph{one-edge} splitting if it has exactly one $G$-orbit of edges. Two $\mathcal{Z}$-splittings are \emph{equivalent} if there exists a $G$-equivariant homeomorphism between them. Given two $(G,\mathcal{F})$-trees $T$ and $T'$, a map $f:T\to T'$ is \emph{alignment-preserving} if the $f$-image of every segment in $T$ is a segment in $T'$. If there exists a $G$-equivariant alignment-preserving map from $T$ to $T'$, we say that $T$ is a \emph{refinement} of $T'$. The \emph{cyclic splitting graph} $FZ(G,\mathcal{F})$ is the graph whose vertices are the equivalence classes of one-edge $\mathcal{Z}$-splittings of $(G,\mathcal{F})$, two distinct vertices being joined by an edge if the corresponding splittings admit a common refinement. The graph $FZ(G,\mathcal{F})$ admits a natural right action of $\text{Out}(G,\mathcal{F})$, by precomposition of the actions. In \cite{Hor14-6}, we proved hyperbolicity of the graph $FZ(G,\mathcal{F})$. 

\begin{theo} (Horbez \cite[Theorem 3.1]{Hor14-6})
Let $G$ be a countable group, and let $\mathcal{F}$ be a free factor system of $G$. Then the graph $FZ(G,\mathcal{F})$ is Gromov hyperbolic. 
\end{theo}

We also described the Gromov boundary of $FZ(G,\mathcal{F})$. A tree $T\in\overline{\mathcal{O}(G,\mathcal{F})}$ is \emph{$\mathcal{Z}$-compatible} if it is compatible with some $\mathcal{Z}$-splitting of $(G,\mathcal{F})$, and \emph{$\mathcal{Z}$-incompatible} otherwise. It is \emph{$\mathcal{Z}$-averse} if it is not compatible with any $\mathcal{Z}$-compatible tree $T'\in\overline{\mathcal{O}(G,\mathcal{F})}$ (see \cite[Section 5.6.1]{Hor14-6} for examples of $\mathcal{Z}$-incompatible trees that are not $\mathcal{Z}$-averse). We denote by $\mathcal{X}(G,\mathcal{F})$ the subspace of $\overline{\mathcal{O}(G,\mathcal{F})}$ consisting of $\mathcal{Z}$-averse trees. Two trees $T,T'\in\mathcal{X}(G,\mathcal{F})$ are \emph{equivalent}, which we denote by $T\sim T'$, if they are both compatible with a common tree in $\overline{\mathcal{O}(G,\mathcal{F})}$. There is a natural, coarsely well-defined map $\psi:\mathcal{O}(G,\mathcal{F})\to FZ(G,\mathcal{F})$.

\begin{theo} (Horbez \cite[Theorem 0.2]{Hor14-6}) \label{boundary-fz}
Let $G$ be a countable group, and let $\mathcal{F}$ be a free factor system of $G$. Then there exists a unique $\text{Out}(G,\mathcal{F})$-equivariant homeomorphism $$\partial{\psi}:\mathcal{X}(G,\mathcal{F})/{\sim}\to\partial_{\infty} FZ(G,\mathcal{F}),$$ so that for all $T\in\mathcal{X}(G,\mathcal{F})$, and all sequences $(T_i)_{i\in\mathbb{N}}\in \mathcal{O}(G,\mathcal{F})^{\mathbb{N}}$ converging to $T$, the sequence $(\psi(T_i))_{i\in\mathbb{N}}$ converges to $\partial{\psi}(T)$. 
\end{theo}

We also proved that every $\sim$-class of $\mathcal{Z}$-averse trees contains a unique simplex of mixing representatives. A tree $T\in\overline{\mathcal{O}(G,\mathcal{F})}$ is \emph{mixing} if for all finite subarcs $I,J\subseteq T$, there exist $g_1,\dots,g_k\in G$ such that $J\subseteq g_1I\cup\dots\cup g_kI$, and for all $i\in\{1,\dots,k-1\}$, we have $g_iI\cap g_{i+1}I\neq\emptyset$. Two $\mathbb{R}$-trees $T$ and $T'$ are \emph{weakly homeomorphic} if there exist maps $f:T\to T'$ and $g:T'\to T$ that are continuous in restriction to segments, and inverse of each other.

\begin{prop} (Horbez \cite[Proposition 5.3]{Hor14-6})\label{mixing-representative}
For all $T\in\mathcal{X}(G,\mathcal{F})$, there exists a mixing tree $\overline{T}\in\mathcal{X}(G,\mathcal{F})$ onto which all trees $T'\in\mathcal{X}(G,\mathcal{F})$ that are equivalent to $T$ collapse. In addition, any two such trees are $G$-equivariantly weakly homeomorphic. Any tree $T\in\overline{\mathcal{O}(G,\mathcal{F})}$ that is both $\mathcal{Z}$-incompatible and mixing, is $\mathcal{Z}$-averse.
\end{prop}

We also mention the following fact about $\mathcal{Z}$-splittings of $(G,\mathcal{F})$. 

\begin{lemma}(Horbez \cite[Lemma 5.11]{Hor14-5})\label{stab-cyclic}
Let $S$ be a $\mathcal{Z}$-splitting of $(G,\mathcal{F})$. Then every edge stabilizer in $S$ is trivial, or contained in a proper $(G,\mathcal{F})$-free factor.
\end{lemma}

\subsection{Transverse families, transverse coverings, graphs of actions}

Let $T$ be a $(G,\mathcal{F})$-tree. A \emph{transverse family} in $T$ is a $G$-invariant collection $\mathcal{Y}$ of nondegenerate (i.e. nonempty and not reduced to a point) subtrees of $T$, such that for all $Y\neq Y'\in\mathcal{Y}$, the intersection $Y\cap Y'$ contains at most one point.

A \emph{transverse covering} of $T$ is a transverse family $\mathcal{Y}$ in $T$, all of whose elements are closed subtrees of $T$, such that every finite arc in $T$ can be covered by finitely many elements of $\mathcal{Y}$. A transverse covering $\mathcal{Y}$ of $T$ is \emph{trivial} if $\mathcal{Y}=\{T\}$. The \emph{skeleton} of a transverse covering $\mathcal{Y}$ is the bipartite simplicial tree $S$, whose vertex set is $V(S)=V_0(S)\cup\mathcal{Y}$, where $V_0(S)$ is the set of points of $T$ which belong to at least two distinct trees in $\mathcal{Y}$, with an edge between $x\in V_0(S)$ and $Y\in\mathcal{Y}$ whenever $x\in Y$ \cite[Definition 4.8]{Gui04}. 
\\
\\
\indent Let $G$ be a countable group, and $\mathcal{F}$ be a free factor system of $G$. A \emph{$(G,\mathcal{F})$-graph of actions} consists of

\begin{itemize}
\item a metric graph of groups $\mathcal{G}$ (in which we allow some edges to have length $0$), with an isomorphism from $G$ to the fundamental group of $\mathcal{G}$, such that all peripheral subgroups are conjugate into vertex groups of $\mathcal{G}$, and 

\item an isometric action of every vertex group $G_v$ on a $G_v$-tree $T_v$ (possibly reduced to a point), in which all intersections of $G_v$ with peripheral subgroups of $G$ are elliptic, and

\item a point $p_e\in T_{t(e)}$ fixed by $i_e(G_e)\subseteq G_{t(e)}$ for every oriented edge $e$, where $i_e:G_e\to G_{t(e)}$ denotes the inclusion morphism from the edge group $G_e$ into the adjacent vertex group $G_{t(e)}$.
\end{itemize}

A $(G,\mathcal{F})$-graph of actions is \emph{nontrivial} if $\mathcal{G}$ is not reduced to a point. Associated to any $(G,\mathcal{F})$-graph of actions $\mathcal{G}$ is a $(G,\mathcal{F})$-tree $T(\mathcal{G})$. Informally, the tree $T(\mathcal{G})$ is obtained from the Bass--Serre tree of the underlying graph of groups by equivariantly attaching each vertex tree $T_v$ at the corresponding vertex $v$, an incoming edge being attached to $T_v$ at the prescribed attaching point. The reader is referred to \cite[Proposition 3.1]{Gui98} for a precise description of the tree $T(\mathcal{G})$. We say that a $(G,\mathcal{F})$-tree $T$ \emph{splits as a $(G,\mathcal{F})$-graph of actions} if there exists a $(G,\mathcal{F})$-graph of actions $\mathcal{G}$ such that $T=T({\mathcal{G}})$.

\begin{prop} (Guirardel \cite[Lemma 1.5]{Gui08})\label{skeleton}
A $(G,\mathcal{F})$-tree splits as a nontrivial $(G,\mathcal{F})$-graph of actions if and only if it admits a nontrivial transverse covering.
\end{prop}

Knowing that a $(G,\mathcal{F})$-tree $T$ is compatible with a simplicial $(G,\mathcal{F})$-tree $S$ provides a nontrivial transverse covering of $T$, defined in the following way (see the discussion in \cite[Section 4.7]{Hor14-6}). Since $T$ and $S$ are compatible, their length functions sum up to the length function of a $(G,\mathcal{F})$-tree, denoted by $T+S$, which comes with $1$-Lipschitz alignment-preserving maps $\pi_T:T+S\to T$ and $\pi_S:T+S\to S$, see \cite[Section 3.2]{GL10-2}. Then the family $\mathcal{Y}$ made of all nondegenerate $\pi_S$-preimages of vertices of $S$, and of the closures of $\pi_S$-preimages of open edges of $S$, is a transverse covering of $T+S$. Its image $\pi_T(\mathcal{Y})$ is a nontrivial transverse covering of $T$. 

We now mention a result, due to Levitt \cite{Lev94}, which gives a canonical way of splitting any very small $(G,\mathcal{F})$-tree as a $(G,\mathcal{F})$-graph of actions, whose vertex actions have dense orbits. 

\begin{prop} (Levitt \cite{Lev94})\label{Levitt}
Every $(G,\mathcal{F})$-tree $T\in\overline{\mathcal{O}(G,\mathcal{F})}$ splits uniquely as a $(G,\mathcal{F})$-graph of actions, all of whose vertex trees have dense orbits for the action of their stabilizer (they might be reduced to points), and all of whose edges have positive length, and have either trivial, or maximally-cyclic and nonperipheral stabilizer.
\end{prop}

We call this splitting the \emph{Levitt decomposition} of $T$ as a graph of actions. We note in particular that if $T\in\overline{\mathcal{O}(G,\mathcal{F})}$ is a very small $(G,\mathcal{F})$-tree, and $H\subseteq G$ is a subgroup of $G$ of finite Kurosh rank, then the $H$-minimal subtree of $T$ admits a Levitt decomposition.

\begin{lemma}\label{nonsimplicial}
Let $T\in\overline{\mathcal{O}(G,\mathcal{F})}$ be a tree with dense orbits. Let $\mathcal{Y}$ be a transverse family in $T$, and let $Y\in\mathcal{Y}$. If $\text{rk}_K(\text{Stab}(Y))<+\infty$, then the action of $\text{Stab}(Y)$ on $Y$ has dense orbits. If $\text{Stab}(Y)$ is contained in a proper $(G,\mathcal{F})$-free factor $H$, then the $H$-minimal subtree of $T$ is not a Grushko $(H,\mathcal{F}_H)$-tree. 
\end{lemma}

\begin{proof}
Assume that one of the conclusions of the lemma fails. Then $Y$ has a nontrivial simplicial part, which contains a simplicial edge $e$. There is a finite number of $G$-orbits of directions at branch points in $T$ \cite[Corollary 4.8]{Hor14-5}. As $T$ has dense orbits, the arc $e$ contains two distinct branch points $x$ and $x'$ of $T$, and two directions $d$ (resp. $d'$) at $x$ (resp. $x'$), such that there exists $g\in G\smallsetminus\{1\}$ with $gd=d'$. In particular, the intersection $gY\cap Y$ is nondegenerate (i.e. nonempty and not reduced to a point). As $\mathcal{Y}$ is a transverse family, this implies that $g\in\text{Stab}(Y)$. So $ge$ is a simplicial edge of $Y$ that meets $e$, and therefore $ge=e$. This implies that $T$ contains an arc with nontrivial stabilizer, which is impossible because $T$ has dense orbits \cite[Proposition 4.17]{Hor14-5}.
\end{proof}

\subsection{Trees of surface type}

\begin{de}
A tree $T\in\overline{\mathcal{O}(G,\mathcal{F})}$ is \emph{of surface type} if it admits a transverse covering by trees that are either simplicial arcs, or are dual to arational measured foliations on compact $2$-orbifolds.
\end{de}

\begin{prop}(Horbez \cite[Proposition 5.10]{Hor14-5})\label{outer-limits}
Let $T$ be a minimal, very small $(G,\mathcal{F})$-tree of surface type, and let $\mathcal{Y}$ be the associated transverse covering of $T$. Then either 
\begin{itemize}
\item there exists an element of $G$, represented by a boundary curve of one of the orbifolds dual to a tree in $\mathcal{Y}$, that is nonperipheral, and not conjugate into any edge group of the skeleton of $\mathcal{Y}$, or 
\item the tree $T$ splits as a $(G,\mathcal{F})$-graph of actions over a one-edge $(G,\mathcal{F})$-free splitting $S$, such that all stabilizers of subtrees in $\mathcal{Y}$ dual to arational foliations on compact $2$-orbifolds are elliptic in $S$.
\end{itemize}
\end{prop}

\begin{prop}(Horbez \cite[Lemma 5.8]{Hor14-5})\label{surface-type}
Let $T\in\overline{\mathcal{O}(G,\mathcal{F})}$. If there exists a subgroup $H\subseteq G$ that is elliptic in $T$, and not contained in any proper $(G,\mathcal{F})$-free factor, then $T$ is of surface type.
\end{prop}

\section{Sporadic cases}\label{sec-2}

Let $G$ be a countable group, and let $\mathcal{F}$ be a free factor system of $G$. We say that $(G,\mathcal{F})$ is \emph{sporadic} if either $G=G_1\ast G_2$ and $\mathcal{F}=\{[G_1],[G_2]\}$, or $G=G_1\ast$ and $\mathcal{F}=\{[G_1]\}$. Otherwise $(G,\mathcal{F})$ is \emph{nonsporadic}. We noticed in \cite[Corollary 5.8]{Hor14-6} that the graph $FZ(G,\mathcal{F})$ is unbounded if and only if $(G,\mathcal{F})$ is nonsporadic. Given a group $A$, we denote by $Z(A)$ its center. The following propositions, which describe $\text{Out}(G,\mathcal{F}^{(t)})$ when $(G,\mathcal{F})$ is sporadic, are particular cases of Levitt's work about automorphisms of graphs of groups \cite{Lev04}.

\begin{prop}\label{sporadic-1}
Let $G_1$ and $G_2$ be nontrivial countable groups. Then $\text{Out}(G_1\ast G_2,\{[G_1],[G_2]\}^{(t)})$ is isomorphic to $G_1/Z(G_1)\times G_2/Z(G_2)$.
\end{prop}

\begin{prop}\label{sporadic-2}
Let $G_1$ be a countable group. Then $\text{Out}(G_1\ast,\{[G_1]\}^{(t)})$ has a subgroup of index $2$ that is isomorphic to $(G_1\times G_1)/Z(G_1)$, where $Z(G_1)$ sits as a subgroup of $G_1\times G_1$ via the diagonal inclusion map.
\end{prop}

\section{Stabilizers of trees in $\overline{\mathcal{O}(G,\mathcal{F})}$}\label{sec-3}

Let $G$ be a countable group, and let $\mathcal{F}$ be a free factor system of $G$. Given $T\in\overline{\mathcal{O}(G,\mathcal{F})}$ (resp. $[T]\in\overline{P\mathcal{O}(G,\mathcal{F})}$), we denote by $\text{Out}(T)$ (resp. $\text{Out}([T])$) the subgroup of $\text{Out}(G,\mathcal{F}^{(t)})$ consisting of those automorphisms that fix $T$ (resp. $[T]$). Notice that $\text{Out}(T)$ sits inside $\text{Out}([T])$ as a normal subgroup. There is a natural morphism $$\lambda:\text{Out}([T])\to\mathbb{R}_+^{\ast},$$ where $\lambda(\Phi)$ is defined as the unique real number such that $T.\Phi=\lambda(\Phi)T$. The kernel of $\lambda$ is equal to $\text{Out}(T)$, so $\text{Out}([T])$ is an abelian extension of $\text{Out}(T)$. One can actually show that the image of $\lambda$ is a cyclic subgroup of $\mathbb{R}_+^{\ast}$ \cite{GL14-2}. 
\\
\\
\indent In \cite[Corollary 3.5]{Hor14-5}, we proved the following about point stabilizers of trees in $\overline{\mathcal{O}(G,\mathcal{F})}$.

\begin{prop}(Horbez \cite[Corollary 3.5]{Hor14-5})\label{per-finite}
Let $T\in\overline{\mathcal{O}(G,\mathcal{F})}$ be a tree with trivial arc stabilizers. Then there are finitely many orbits of points in $T$ with nontrivial stabilizer. For all $v\in T$, we have $\text{rk}_K(\text{Stab}(v))<\text{rk}_K(G,\mathcal{F})$.
\end{prop}

Let $T\in\overline{\mathcal{O}(G,\mathcal{F})}$ be a tree with trivial arc stabilizers. Let $V$ be the collection of $G$-orbits of points with nontrivial stabilizer in $T$. Let $\{G_v\}_{v\in V}$ be a set of representatives of the $G$-conjugacy classes of point stabilizers in $T$. We define $\text{Out}(T,\{[G_v]\}_{v\in V}^{(t)})$ to be the subgroup of $\text{Out}(T)$ made of those automorphisms that are a conjugation by an element of $G$ in restriction to every point stabilizer of $T$.

\begin{theo} (Guirardel--Levitt \cite{GL14-2})\label{Guirardel-Levitt}
Let $T\in\overline{\mathcal{O}(G,\mathcal{F})}$ be a tree with trivial arc stabilizers. Let $V$ be the collection of orbits of points in $T$ with nontrivial stabilizer, and let $\{G_v\}_{v\in V}$ be the collection of point stabilizers in $T$. Then $\text{Out}(T,\{[G_v]\}^{(t)})$ has a finite index subgroup $\text{Out}^0(T,\{[G_v]\}^{(t)})$ which admits an injective morphism $$\text{Out}^0(T,\{[G_v]\}^{(t)})\hookrightarrow\prod_{v\in V}G_v^{d_v}/Z(G_v),$$ where $d_v$ denotes the degree of $v$ in $T$, and $Z(G_v)$ denotes the center of $G_v$, and $Z(G_v)$ sits as a diagonal subgroup of $G_v^{d_v}$ via the diagonal inclusion map. 
\end{theo}

A consequence of Guirardel and Levitt's theorem is the following fact.

\begin{cor}\label{Tits-stabilizer}
Let $T\in\overline{\mathcal{O}(G,\mathcal{F})}$ be a tree with trivial arc stabilizers. Let $V$ be the collection of orbits of points in $T$ with nontrivial stabilizer, and let $\{G_v\}_{v\in V}$ be the collection of point stabilizers in $T$. If $G$ satisfies the Tits alternative, then $\text{Out}(T,\{[G_v]\}^{(t)})$ satisfies the Tits alternative.
\end{cor}

\section{Arational $(G,\mathcal{F})$-trees}\label{sec-4}

Let $G$ be a countable group, and let $\mathcal{F}:=\{[G_1],\dots,[G_k]\}$ be a free factor system of $G$. We recall that a $(G,\mathcal{F})$-free factor is \emph{proper} if it is nonperipheral (in particular nontrivial), and not equal to $G$.

\begin{de}
A $(G,\mathcal{F})$-tree $T\in\overline{\mathcal{O}(G,\mathcal{F})}$ is \emph{arational} if $T\in\partial \mathcal{O}(G,\mathcal{F})$ and for every proper $(G,\mathcal{F})$-free factor $H\subset G$, the factor $H$ is not elliptic in $T$, and the $H$-minimal subtree $T_H$ of $T$ is a Grushko $(H,\mathcal{F}_{H})$-tree, i.e. the action of $H$ on $T_H$ is simplicial and relatively free.
\end{de}

We denote by $\mathcal{AT}(G,\mathcal{F})$ the subspace of $\overline{\mathcal{O}(G,\mathcal{F})}$ consisting of arational $(G,\mathcal{F})$-trees. 

\subsection{Arational surface $(G,\mathcal{F})$-trees}

We describe a way of constructing arational $(G,\mathcal{F})$-trees, illustrated in Figure \ref{fig-arational}. We first need the following fact. 

\begin{prop}\label{finite-index}
Let $T$ be a tree dual to an arational measured foliation on a compact $2$-orbifold $\mathcal{O}$ with conical singularities, and let $H\subseteq\pi_1(\mathcal{O})$ be a finitely generated subgroup of $\pi_1(\mathcal{O})$ of infinite index. Then the $H$-minimal subtree of $T$ is simplicial.
\end{prop}  

A proof of Proposition \ref{finite-index} appears in \cite{Rey11-2} in the case where $\mathcal{O}$ is a compact surface, and it adapts to the case where $\mathcal{O}$ is a $2$-orbifold. Proposition \ref{finite-index} can also be deduced from the surface case by using Selberg's Lemma, which states that $\pi_1(\mathcal{O})$ has a finite-index subgroup which is the fundamental group of a compact surface.
\\
\\
\noindent Let $\mathcal{O}$ be a compact $2$-orbifold of genus $g$ with conical singularities, having $s+1$ boundary curves $b_0,b_1,\dots,b_s$, and $q$ conical points $b_{s+1},\dots,b_{s+q}$, equipped with an arational measured foliation. We build a graph of groups $\mathcal{G}'$ in the following way. One of the vertex groups of $\mathcal{G}'$ is the fundamental group of the orbifold $\mathcal{O}$, and the others are the peripheral subgroups $G_i$. For all $i\in\{1,\dots,s+q\}$, we choose $j_i\in\{1,\dots,k\}$, and an element $g_i\in G_{j_i}$, of same order as $b_i$. We put an edge between the vertex of $\mathcal{G}'$ associated to $\mathcal{O}$ and the vertex associated to $G_{j_i}$, and we amalgamate $b_i$ with $g_i$. Choices are made in such a way that the graph $\mathcal{G}'$ we get is connected. We then define a graph of groups $\mathcal{G}$ as the minimal subgraph of groups of $\mathcal{G}'$, i.e. $\mathcal{G}$ is obtained from $\mathcal{G}'$ by removing vertices $G_j$ with exactly one incident edge, and such that $G_i$ is cyclic and generated by $b_i$. Notice that the element of $\pi_1(\mathcal{O})$ corresponding to the boundary curve $b_0$ does not fix any edge in $\mathcal{G}$. The fundamental group of $\mathcal{G}$ is isomorphic to $G:=G_1\ast\dots\ast G_k\ast F_N$, where $N=2g+b_1(\mathcal{G})$ if $\mathcal{O}$ is orientable, and $N=g+b_1(\mathcal{G})$ if $\mathcal{O}$ is nonorientable. 

\begin{figure}
\begin{center}
\input{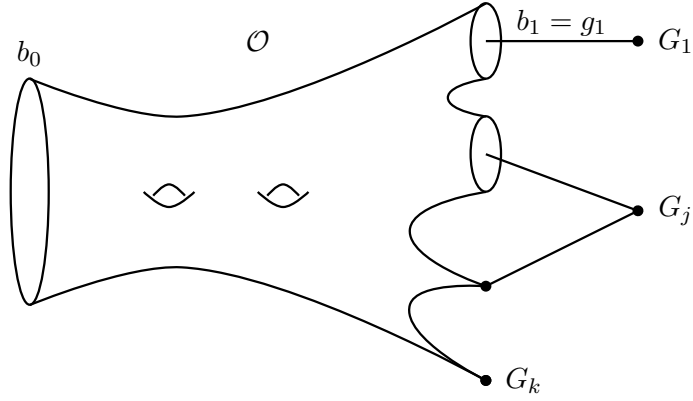}
\caption{An arational surface $(G,\mathcal{F})$-tree.}
\label{fig-arational}
\end{center}
\end{figure}

Dual to the foliation on $\mathcal{O}$ is a $\pi_1(\mathcal{O})$-tree $Y$. We form a graph of actions over $\mathcal{G}$: vertex trees are the $\pi_1(\mathcal{O})$-tree $Y$, and a trivial $G_i$-tree for all $i\in\{1,\dots,k\}$, attaching points in $Y$ are the points fixed by the $b_i$'s, and edges have length $0$. We denote by $T$ the $(G,\mathcal{F})$-tree defined in this way.

\begin{de}
A $(G,\mathcal{F})$-tree obtained by the above construction is called an \emph{arational surface} $(G,\mathcal{F})$-tree.
\end{de}

We claim that the $(G,\mathcal{F})$-tree $T$ we have built is an arational $(G,\mathcal{F})$-tree, which justifies our terminology. 

We start by making the following remarks: all point stabilizers in $Y$ are peripheral, except $b_0$. The element $b_0$ is not contained in any proper $(G,\mathcal{F})$-free factor. Indeed, otherwise, there would exist a $(G,\mathcal{F})$-free splitting $S$ in which $b_0$ is elliptic, and all other boundary components of $\mathcal{O}$ would also be elliptic in $S$ because they are peripheral. The splitting $S$ would then restrict to a free splitting of $\pi_1(\mathcal{O})$ in which all boundary components are elliptic. Such a splitting does not exist, so we have reached a contradiction. 

Let now $H$ be a proper $(G,\mathcal{F})$-free factor. Assume towards a contradiction that the $H$-minimal subtree of $T$ is not a Grushko $(H,\mathcal{F}_H)$-tree. The action of $H$ on $T$ is relatively free because $b_0$ is not contained in any proper $(G,\mathcal{F})$-free factor, so the action of $H$ is not discrete. The transverse covering of $T$ made of the translates of the $\pi_1(\mathcal{O})$-minimal subtree of $T$ induces a transverse covering of the $H$-minimal subtree of $T$, whose nontrivial elements are $H\cap\pi_1(\mathcal{O})^g$-trees, for some $g\in G$. Therefore, there exists a conjugate $H^g$ of $H$ so that $H^g\cap\pi_1(\mathcal{O})\neq\{e\}$, and the action of $H^g\cap\pi_1(\mathcal{O})$ on its minimal subtree is non-simplicial. By Proposition \ref{finite-index}, this implies that $H^g\cap\pi_1(\mathcal{O})$ has finite index in $\pi_1(\mathcal{O})$. As $H$ is elliptic in a $(G,\mathcal{F})$-free splitting $S$, so is $\pi_1(\mathcal{O})$: the group $\pi_1(\mathcal{O})$ fixes a unique point in $S$. All other vertex stabilizers of the Bass--Serre tree $S_0$ of $\mathcal{G}$ are peripheral, so each of them fixes a unique point in $S$. Since edge stabilizers of $S_0$ are peripheral, the stabilizers of any two adjacent vertices in $S_0$ contain a common peripheral element. This implies that they have the same fixed point in $S$, because no peripheral element fixes an arc in $S$. Therefore, all vertex groups of $S_0$ fix the same point in $S$. Hence $G$ is elliptic in $S$, a contradiction.

\subsection{A classification result}

The goal of this section is to provide a classification result for trees in $\overline{\mathcal{O}(G,\mathcal{F})}$. When $T\in\overline{\mathcal{O}(G,\mathcal{F})}$ is not arational, a proper $(G,\mathcal{F})$-free factor is a \emph{dynamical free factor} for $T$ if it acts with dense orbits on its minimal subtree but does not fix any point in $T$. The following proposition is an extension of \cite[Proposition 2.1]{Hor14-4} to the context of $(G,\mathcal{F})$-trees.

\begin{prop}\label{classification}
Let $G$ be a countable group, and let $\mathcal{F}$ be a free factor system of $G$. Then for all $(G,\mathcal{F})$-trees $T\in\overline{\mathcal{O}(G,\mathcal{F})}$, either
\begin{itemize}
\item we have $T\in\mathcal{O}(G,\mathcal{F})$, or
\item the tree $T$ is arational, or
\item the tree $T$ has a dynamical free factor, or
\item the tree $T$ has no dynamical free factor, and there exists $x\in T$ whose stabilizer is nonperipheral, and is contained in a proper $(G,\mathcal{F})$-free factor.
\end{itemize}
\end{prop} 

\begin{lemma}\label{dense}
Let $T$ be a $(G,\mathcal{F})$-tree with trivial arc stabilizers. Let $H\subseteq G$ be a nonperipheral subgroup of $G$ that is contained in a proper $(G,\mathcal{F})$-free factor. If $H$ fixes a point in $T$, then $T$ is not arational. If the $H$-minimal subtree of $T$ is not simplicial, then $T$ has a dynamical proper free factor.   
\end{lemma}

\begin{proof}
Let $F$ be a proper $(G,\mathcal{F})$-free factor that contains $H$. If $H$ fixes a point in $T$, then the action of $F$ is not relatively free, which implies that $T$ is arational.  

By Proposition \ref{Levitt}, the $F$-minimal subtree $T_F$ of $T$ splits as a graph of actions $\mathcal{G}$ with trivial edge stabilizers, in which all vertex actions have dense orbits (they may be trivial). Vertex groups of $\mathcal{G}$ are $(G,\mathcal{F})$-free factors. If the $H$-minimal subtree of $T$ is non-simplicial, then $T_F$ is non-simplicial, so one of the vertex groups of $\mathcal{G}$ is a dynamical proper $(G,\mathcal{F})$-free factor of $T$. 
\end{proof}

\begin{lemma}\label{classif}
Let $T$ be a $(G,\mathcal{F})$-tree with trivial arc stabilizers. Assume that $T$ is not relatively free. Then either
\begin{itemize}
\item the tree $T$ is an arational surface tree (in particular, all elliptic subgroups in $T$ are either cyclic or peripheral), or
\item the tree $T$ has a dynamical proper free factor, or
\item there exists a nonperipheral point stabilizer in $T$ that is contained in a proper $(G,\mathcal{F})$-free factor, and all noncyclic, nonperipheral point stabilizers in $T$ are contained in proper $(G,\mathcal{F})$-free factors.
\end{itemize}
\end{lemma}

\begin{proof}
If all elliptic subgroups of $T$ are contained in proper $(G,\mathcal{F})$-free factors, then the last assertion holds. Otherwise, Lemma \ref{surface-type} implies that $T$ is a tree of surface type. Let $\mathcal{Y}$ be the transverse covering of $T$ provided by the definition of trees of surface type. 

If the stabilizer of a tree in $\mathcal{Y}$ dual to an arational measured foliation on a compact $2$-orbifold is contained in a proper $(G,\mathcal{F})$-free factor, then the second assertion holds by Lemma \ref{dense}. This occurs in particular if the skeleton of $\mathcal{Y}$ contains an edge with trivial stabilizer, so we can assume that this is not the case.
 
Otherwise, Proposition \ref{outer-limits} implies that there exists an element of $G$, represented by a boundary curve $c$ of an orbifold $\Sigma$ dual to a tree in $\mathcal{Y}$, that is nonperipheral, and not conjugate into any edge group of the skeleton of $\mathcal{Y}$. If the transverse covering $\mathcal{Y}$ contains at least two orbits of nondegenerate trees, then an arc on $\Sigma$ whose endpoints lie on $c$ determines a $(G,\mathcal{F})$-free splitting, in which the other orbifold groups are elliptic, and hence contained in a proper $(G,\mathcal{F})$-free factor. Again, the second assertion of the lemma holds. Similarly, if there exists a point in $T$, whose stabilizer is nonperipheral and not conjugate to $c$, then the third conclusion of the lemma holds.

In the remaining case, the skeleton of $\mathcal{Y}$ contains a single orbit of vertices $v$ associated to a tree $T_0$ dual to an arational lamination on a $2$-orbifold $\mathcal{O}$. All vertices $v'$ adjacent to $v$ have stabilizer isomorphic to some $G_i$. The edge joining $v'$ to $v$ has nontrivial stabilizer, so it is attached in $T_0$ to a point corresponding to a boundary curve or a conical point of $\mathcal{O}$. In addition, all boundary curves (and conical points) of $\Sigma$ distinct from $c$ are peripheral. This implies that $T$ is an arational surface $(G,\mathcal{F})$-tree. 
\end{proof}

\begin{proof}[Proof of Proposition \ref{classification}]
Let $T\in\partial{\mathcal{O}(G,\mathcal{F})}$ be a tree which is not arational, and has no dynamical proper $(G,\mathcal{F})$-free factor. Then the $G$-action on $T$ is not relatively free. If $T$ has trivial arc stabilizers, then the conclusion follows from Lemma \ref{classif}.   

We now assume that $T$ contains an arc $e$ with nontrivial stabilizer, and let $S$ be the very small simplicial $(G,\mathcal{F})$-tree obtained by collapsing to points all vertex trees in the Levitt decomposition of $T$ as a graph of actions (Proposition \ref{Levitt}). The stabilizer $G_e$ of $e$ in $T$ also stabilizes an edge in $S$. By Lemma \ref{stab-cyclic}, the group $G_e$ is contained in a proper $(G,\mathcal{F})$-free factor, and in addition $G_e$ is nonperipheral because $T$ is very small. We can thus choose for $x$ some interior point of $e$. 
\end{proof}

\subsection{Arational $(G,\mathcal{F})$-trees are $\mathcal{Z}$-averse.}

\begin{prop}\label{ATX}
Let $G$ be a countable group, and let $\mathcal{F}$ be a free factor system of $G$. Then $\mathcal{AT}(G,\mathcal{F})\subseteq\mathcal{X}(G,\mathcal{F})$.
\end{prop}

\begin{proof}
In view of Proposition \ref{mixing-representative}, it is enough to show that any tree $T\in\mathcal{AT}(G,\mathcal{F})$ is both $\mathcal{Z}$-incompatible and mixing. This will be done in Lemmas \ref{at-incompatible} and \ref{at-mixing}. 
\end{proof}

\begin{lemma}\label{at-incompatible}
Every arational $(G,\mathcal{F})$-tree is $\mathcal{Z}$-incompatible.
\end{lemma}

\begin{proof}[Proof of Lemma \ref{at-incompatible}]
Let $T\in\overline{\mathcal{O}(G,\mathcal{F})}$ be a $\mathcal{Z}$-compatible tree. It follows from the discussion below Proposition \ref{skeleton} that $T$ splits as a $(G,\mathcal{F})$-graph of actions $\mathcal{G}$, whose edge groups are either trivial, or cyclic and nonperipheral. If $\mathcal{G}$ contains a nontrivial edge group $G_e$, then $G_e$ must be elliptic in $T$. The group $G_e$ is contained in a proper $(G,\mathcal{F})$-free factor $F$ (Lemma \ref{stab-cyclic}), and it is nonperipheral because $T$ is very small. By Lemma \ref{dense}, the tree $T$ is not arational. 

If all edge groups of $\mathcal{G}$ are trivial, then all vertex groups of $\mathcal{G}$ are proper $(G,\mathcal{F})$-free factors. If all vertex actions of $\mathcal{G}$ are Grushko $(G_v,\mathcal{F}_{G_v})$-trees, then $T$ is simplicial, with trivial edge stabilizers. So either $T$ is a Grushko $(G,\mathcal{F})$-tree, or some vertex stabilizer of $T$ is a proper free factor that acts elliptically on $T$. In both cases, the tree $T$ is not arational.
\end{proof}

The following lemma was proved by Reynolds in \cite[Proposition 8.3]{Rey12} in the case of $F_N$-trees in the closure of Culler and Vogtmann's outer space.

\begin{lemma}\label{at-mixing}
Every arational $(G,\mathcal{F})$-tree is mixing.
\end{lemma}

Let $T,\overline{T}\in\overline{\mathcal{O}(G,\mathcal{F})}$. We say that $T$ \emph{collapses} onto $\overline{T}$ if there exists a $G$-equivariant map $p:T\to \overline{T}$ that sends segments of $T$ onto segments of $\overline{T}$. The following lemma follows from work by Guirardel and Levitt \cite{GL14-2}, together with \cite[Proposition 5.17]{Hor14-6}.

\begin{lemma}\label{mixing-collapse}
Let $T\in\overline{\mathcal{O}(G,\mathcal{F})}$ be a tree with dense $G$-orbits, and let $Y\varsubsetneq T$ be a proper subtree, such that for all $g\in G$, either $gY=Y$, or $gY\cap Y=\emptyset$. Then either $T$ is compatible with a $(G,\mathcal{F})$-free splitting, or else $T$ collapses onto a mixing tree $\overline{T}\in\overline{\mathcal{O}(G,\mathcal{F})}$ in which $\text{Stab}(Y)$ is elliptic. 
\qed
\end{lemma}

\begin{proof}[Proof of Lemma \ref{at-mixing}]
Let $T\in\mathcal{AT}(G,\mathcal{F})$. Then $T$ has dense orbits, otherwise any simplicial edge in $T$ would be dual to a $\mathcal{Z}$-splitting that is compatible with $T$, contradicting Lemma \ref{at-incompatible}. Assume towards a contradiction that $T$ is not mixing, and let $I\subset T$ be a segment. Define $Y_I$ to be the subtree of $T$ consisting of all points $x\in T$ such that there exists a finite set of elements $\{g_0=e,g_1,\dots,g_r\}\subset G$, with $x\in g_rI$, and $g_iI\cap g_{i+1}I\neq\emptyset$ for all $i\in\{0,\dots,r-1\}$. Then for all $g\in G$, we either have $gY_I=Y_I$, or $gY_I\cap Y_I=\emptyset$.

As $T$ is not mixing, there exists a nondegenerate arc $I\subset T$ such that $Y_I$ is a proper subtree of $T$. By Lemma \ref{mixing-collapse}, either $T$ is compatible with a $(G,\mathcal{F})$-free splitting, or else $T$ collapses onto a mixing tree $\overline{T}\in\overline{\mathcal{O}(G,\mathcal{F})}$, in which $\text{Stab}(Y_I)$ is elliptic. The first case is excluded by Lemma \ref{at-incompatible}, so we assume that we are in the second case. As $T$ has dense orbits, the stabilizer $\text{Stab}(Y_I)$ is not cyclic by Lemma \ref{nonsimplicial}. It thus follows from Lemma \ref{classif} that either $\overline{T}$ has a dynamical proper $(G,\mathcal{F})$-free factor $F$ (if the second situation of Lemma \ref{classif} occurs), or else $\text{Stab}(Y_I)$ is contained in a proper $(G,\mathcal{F})$-free factor (if the third situation of this lemma occurs). In the first case, the $F$-minimal subtree $T_F$ of $T$ cannot be a Grushko $(F,\mathcal{F}_F)$-tree, because $T_F$ collapses to a nontrivial tree with dense orbits in $\overline{T}$. This contradicts arationality of $T$. Hence the second case occurs, i.e. $\text{Stab}(Y_I)$ is contained in a proper $(G,\mathcal{F})$-free factor $F$. By Lemma \ref{nonsimplicial}, the $F$-minimal subtree of $T$ is not a Grushko $(F,\mathcal{F}_{F})$-tree, again contradicting arationality of $T$.
\end{proof}

\subsection{Finite sets of reducing factors associated to non-arational $(G,\mathcal{F})$-trees}

Given a $(G,\mathcal{F})$-tree $T\in\overline{P\mathcal{O}(G,\mathcal{F})}$, we denote by $\text{Dyn}(T)$ the set of minimal (with respect to inclusion) conjugacy classes of dynamical proper $(G,\mathcal{F})$-free factors for $T$. We denote by $\text{Ell}(T)$ the set of nonperipheral conjugacy classes of point stabilizers in $T$. Recall that given a subgroup $H\subseteq G$, we denote by $\text{Fill}(H)$ the smallest $(G,\mathcal{F})$-free factor that contains $H$. For all $\Phi\in\text{Out}(G,\mathcal{F}^{(t)})$, we have $\Phi \text{Dyn}(T)=\text{Dyn}(\Phi T)$, and $\Phi \text{Fill}({\text{Ell}}(T))=\text{Fill}(\text{Ell}(\Phi T))$. It follows from Proposition \ref{per-finite} that $\text{Ell}(T)$ is finite, we will now show that $\text{Dyn}(T)$ is also finite.

\begin{prop}\label{dyn-finite}
For all $T\in\overline{P\mathcal{O}(G,\mathcal{F})}$, the set $\text{Dyn}(T)$ is finite.
\end{prop}

Let $T\in\overline{\mathcal{O}(G,\mathcal{F})}$. A finite subtree $K\subseteq T$ (i.e. the convex hull of a finite set of points) is a \emph{supporting subtree} of $T$ if for all segments $J\subseteq T$, there exists a finite subset $\{g_1,\dots g_r\}\subseteq G$ such that $J\subseteq g_1K\cup\dots\cup g_rK$.

\begin{lemma}\label{cost-1}
Let $T\in\overline{\mathcal{O}(G,\mathcal{F})}$ be a tree with dense orbits. For all $\epsilon>0$, there exists a finite supporting subtree $K\subseteq T$ whose volume is at most $\epsilon$.
\end{lemma}

\begin{proof}
As $T$ has dense orbits, it follows from \cite[Theorem 5.3]{Hor14-5} that there exists a sequence $(T_n)_{n\in\mathbb{N}}\in\mathcal{O}(G,\mathcal{F})^{\mathbb{N}}$, such that the volume of the quotient graph $T_n/G$ converges to $0$, and for all $n\in\mathbb{N}$, there exists a $1$-Lipschitz $G$-equivariant map $f_n:T_n\to T$. Letting $K_n$ be a finite supporting subtree of $T_n$, with volume converging to $0$ as $n$ goes to $+\infty$, the images $f_n(K_n)$ are finite supporting subtrees of $T$ whose volumes converge to $0$.
\end{proof}

Given a finite system $S=(F,A)$ of partial isometries of a finite forest $F$, we define $m(S)$ as the volume of $F$, and $d(S)$ as the sum of the volumes of the domains of the partial isometries in $A$. We say that $S$ has \emph{independent generators} if no reduced word in the partial isometries in $A$ and their inverses defines a partial isometry of $F$ that fixes a nondegenerate arc. Gaboriau, Levitt and Paulin have shown in \cite[Proposition 6.1]{GLP94} that if $S$ has independent generators, then $m(S)-d(S)\ge 0$. The following proposition is a generalization of \cite[Lemma 3.10]{Rey11-2} to the context of $(G,\mathcal{F})$-trees. 

\begin{prop}\label{stab}
Let $T\in\overline{\mathcal{O}(G,\mathcal{F})}$ be a tree with dense orbits, and let $H$ be a $(G,\mathcal{F})$-free factor. Assume that $H$ acts with dense orbits on its minimal subtree $T_H$ in $T$. Then $\text{Stab}(T_H)=H$, and $\{gT_H|g\in G\}$ is a transverse family in $T$.
\end{prop}

\begin{proof}
Let $g\in G\smallsetminus H$. Assume towards a contradiction that $gT_H\cap T_H$ contains a nondegenerate arc $I$ of length $L>0$. Let $\epsilon>0$, with $\epsilon<\frac{L}{2}$. Lemma \ref{cost-1} applied to the $(H,\mathcal{F}_H)$-tree $T_H$ ensures the existence of a finite tree $F_{\epsilon}\subseteq T_H$ of volume smaller than $\epsilon$, such that $I$ is covered by finitely many translates of $F_{\epsilon}$, and we can choose $F_{\epsilon}$ to be disjoint from $I$. We can therefore subdivide $I$ into finitely many subsegments $I_1,\dots,I_k$ such that for all $i\in\{1,\dots,k\}$, there exists $g_i\in H$ with $g_iI_i\subseteq F_{\epsilon}$. Similarly, there exists a finite forest $F'_{\epsilon}\subseteq gT_H$ of volume smaller than $\epsilon$, such that $I$ is covered by finitely many translates of $F'_{\epsilon}$, and again we can choose $F'_{\epsilon}$ to be disjoint from both $I$ and $F_{\epsilon}$ in $T$. We similarly have a subdivision $I'_1,\dots, I'_l$ of $I$, and an element $g'_j\in H^g$ for each $j\in\{1,\dots,l\}$, so that $g'_jI'_j\subseteq F'_{\epsilon}$. We build a system of partial isometries $S$ on the forest $I\cup F_{\epsilon}\cup F'_{\epsilon}$, with an isometry $\phi_i$ from $I_i$ to $F_{\epsilon}$ corresponding to the action of $g_i$ for all $i\in\{1,\dots,k\}$, and an isometry $\phi'_j$ from $I'_j$ to $F'_{\epsilon}$ corresponding to the action of $g'_j$ for all $j\in\{1,\dots,l\}$. Then $m(S)\le L+2\epsilon$, while $d(S)=2L$. Therefore $m(S)-d(S)<0$, and hence the system of isometries $S$ does not have independent generators \cite[Proposition 6.1]{GLP94}. This means that there exists a reduced word $w$ in the partial isometries $\phi_i$, $\phi'_j$ and their inverses, associated to an element $g\in G$ which fixes an arc in $T$. It follows from the construction of the system of isometries that up to cyclic conjugation, the word $w$ is a concatenation of $2$-letter words of the form $\phi_{i_1}\circ\phi_{i_2}^{-1}$ and $\phi'_{j_1}\circ{\phi'_{j_2}}^{-1}$, with $i_1\neq i_2$ and $j_1\neq j_2$, and these two types of subwords alternate in $w$. So $g$ is of the form $h_1h_2^g\dots h_{s-1}h_{s}^g$, where $h_i\in H$ is a nontrivial element for all $i\in\{1,\dots,s\}$. Since $H$ is a proper $(G,\mathcal{F})$-free factor, and $g\in G\smallsetminus H$, we have $\langle H,H^g\rangle = H\ast H^g$, so $g\neq e$. This contradicts the fact that $T$ has dense orbits, and hence trivial arc stabilizers \cite[Proposition 4.17]{Hor14-5}. Therefore, for all $g\in G\smallsetminus H$, the intersection $gT_H\cap T_H$ consists in at most one point. This implies that $\text{Stab}(T_H)=H$, and that $\{gT_H\}_{g\in G}$ is a transverse family in $T$. 
\end{proof}

\begin{prop}\label{minimal-transverse}
Let $T\in\overline{\mathcal{O}(G,\mathcal{F})}$ be a tree with dense orbits. Then the collection $\{gT_H|H\in\text{Dyn}(T),g\in G\}$ is a transverse family in $T$.
\end{prop}

\begin{proof}
Let $H,H'\in\text{Dyn}(T)$, and assume that $T_H\cap T_{H'}$ contains a nondegenerate arc. By Proposition \ref{stab}, since $H$ and $H'$ are proper $(G,\mathcal{F})$-free factors, we have $\text{Stab}(T_H)=H$ and $\text{Stab}(T_{H'})=H'$. The collections $\{gT_H\}_{g\in G}$ and $\{gT_{H'}\}_{g\in G}$ are transverse families in $T$ (Proposition \ref{stab}), hence so is the collection of nondegenerate intersections of the form $gT_H\cap g'T_{H'}$ for $g,g'\in G$. If $g\in G$ stabilizes $T_H\cap T_{H'}$, then $gT_H\cap T_H$ and $gT_{H'}\cap T_{H'}$ both contain a nondegenerate arc, and hence $gT_H=T_H$ and $gT_{H'}=T_{H'}$. So we have $\text{Stab}(T_H\cap T_{H'})=\text{Stab}(T_H)\cap\text{Stab}(T_{H'})=H\cap H'$. By Lemma \ref{nonsimplicial}, the $(G,\mathcal{F})$-free factor $H\cap H'$ acts with dense orbits on the minimal subtree of $T_H\cap T_{H'}$. By minimality of the factors in $\text{Dyn}(T)$, this implies that $H=H'$ and $T_H=T_{H'}$. So $\{gT_H|H\in\text{Dyn}(T),g\in G\}$ is a transverse family in $T$. 
\end{proof}

\begin{proof}[Proof of Proposition \ref{dyn-finite}]
Finiteness of $\text{Dyn}(T)$ for all trees $T\in\overline{P\mathcal{O}(G,\mathcal{F})}$ follows from Proposition \ref{minimal-transverse}, since every transverse family in a tree with dense orbits contains boundedly many orbits of trees (where the bound is given by the number of orbits of directions at branch points in $T$).
\end{proof}

\section{Nonelementary subgroups of $\text{Out}(G,\mathcal{F})$, and a trichotomy for subgroups of $\text{Out}(G,\mathcal{F})$}\label{sec-5}

Let $G$ be a countable group, and let $\mathcal{F}$ be a free factor system of $G$, such that $(G,\mathcal{F})$ is nonsporadic.

\begin{de}
A subgroup $H\subseteq\text{Out}(G,\mathcal{F})$ is \emph{nonelementary} if 

\begin{itemize}
\item it does not preserve any finite set of proper $(G,\mathcal{F})$-free factors, and
\item it does not preserve any finite set of points in $\partial_{\infty} FZ(G,\mathcal{F})$.
\end{itemize}
\end{de}

We now aim at showing that any nonelementary subgroup of $\text{Out}(G,\mathcal{F})$ contains a rank two free subgroup.

\begin{theo}\label{nonelementary}
Let $G$ be a countable group, and let $\mathcal{F}$ be a free factor system of $G$, so that $(G,\mathcal{F})$ is nonsporadic. Then any nonelementary subgroup of $\text{Out}(G,\mathcal{F})$ contains a free subgroup of rank two, generated by two loxodromic isometries of $FZ(G,\mathcal{F})$. 
\end{theo}

As a consequence of Theorem \ref{nonelementary} and of our description of the Gromov boundary of $\partial_{\infty}FZ(G,\mathcal{F})$, we get the following trichotomy for subgroups of $\text{Out}(G,\mathcal{F})$.

\begin{theo}\label{trichotomy}
Let $G$ be a countable group, and let $\mathcal{F}$ be a free factor system of $G$, so that $(G,\mathcal{F})$ is nonsporadic. Then every subgroup of $\text{Out}(G,\mathcal{F})$ either 
\begin{itemize}
\item contains a rank two free subgroup generated by two loxodromic isometries of $FZ(G,\mathcal{F})$, or
\item virtually fixes a tree with trivial arc stabilizers in $\partial P\mathcal{O}(G,\mathcal{F})$, or
\item virtually fixes the conjugacy class of a proper $(G,\mathcal{F})$-free factor.
\end{itemize}
\end{theo} 

\begin{proof}
Let $H$ be a subgroup of $\text{Out}(G,\mathcal{F})$. If $H$ is nonelementary, Theorem \ref{trichotomy} follows from Theorem \ref{nonelementary}. Otherwise, either $H$ virtually fixes the conjugacy class of a proper $(G,\mathcal{F})$-free factor, or $H$ virtually fixes a point $\xi\in\partial_{\infty}FZ(G,\mathcal{F})$. In the latter case, the group $H$ preserves the simplex of length measures in $\overline{P\mathcal{O}(G,\mathcal{F})}$ corresponding to a mixing representative of $\xi$, provided by Proposition \ref{mixing-representative}, and this simplex has finite dimension by \cite[Corollary 5.4]{Gui00} (the extension of Guirardel's result concerning finite dimensionality of this simplex to the case of free products is made possible by the fact that $\overline{P\mathcal{O}(G,\mathcal{F})}$ has finite topological dimension \cite[Theorem 0.2]{Hor14-5}). So $H$ virtually fixes any extremal point of this simplex, which is a tree with trivial arc stabilizers. 
\end{proof}

Our proof of Theorem \ref{nonelementary} uses techniques coming from the theory of random walks on groups. These were already used in \cite{Hor14-3} for giving a new proof of a result of Handel and Mosher \cite{HM09}, which establishes a dichotomy for subgroups of $\text{Out}(F_N)$, namely: every subgroup of $\text{Out}(F_N)$ (finitely generated or not) either contains a fully irreducible automorphism, or virtually fixes the conjugacy class of a proper free factor of $F_N$. All topological spaces will be equipped with their Borel $\sigma$-algebra. Let $\mu$ be a probability measure on $\text{Out}(G,\mathcal{F})$. A probability measure $\nu$ on $\overline{P\mathcal{O}(G,\mathcal{F})}$ is \emph{$\mu$-stationary} if $\mu\ast\nu=\nu$, i.e. for all $\nu$-measurable subsets $E\subseteq \overline{P\mathcal{O}(G,\mathcal{F})}$, we have

\begin{displaymath}
\nu(E)=\sum_{\Phi\in\text{Out}(G,\mathcal{F})}\mu(\Phi)\nu(\Phi^{-1}E).
\end{displaymath} 

\noindent We denote by $P\mathcal{AT}(G,\mathcal{F})$ the image of $\mathcal{AT}(G,\mathcal{F})$ in $\overline{P\mathcal{O}(G,\mathcal{F})}$. Our first goal will be to show that given a probability measure $\mu$ on $\text{Out}(G,\mathcal{F})$, any $\mu$-stationary measure on $\overline{P\mathcal{O}(G,\mathcal{F})}$ is supported on $P\mathcal{AT}(G,\mathcal{F})$. Since $\mathcal{AT}(G,\mathcal{F})\subseteq\mathcal{X}(G,\mathcal{F})$ (Proposition \ref{ATX}), it follows that any $\mu$-stationary measure on $\overline{P\mathcal{O}(G,\mathcal{F})}$ pushes to a $\mu$-stationary measure on $\partial_{\infty} FZ(G,\mathcal{F})$ via the map $\partial\psi$ provided by Theorem \ref{boundary-fz} (this map factors through $\overline{P\mathcal{O}(G,\mathcal{F})}$). We will make use of the following classical lemma, whose proof is based on a maximum principle argument. The following version of the statement appears in \cite[Lemma 3.3]{Hor14-3}. We denote by $gr(\mu)$ the subgroup of $\text{Out}(G,\mathcal{F})$ generated by the support of the measure $\mu$.

\begin{lemma} \label{disjoint-translations} (Ballmann \cite{Bal89})
Let $\mu$ be a probability measure on a countable group $G$, and let $\nu$ be a $\mu$-stationary probability measure on a $G$-space $X$. Let $D$ be a countable $G$-set, and let $\Theta:X\to D$ be a measurable $G$-equivariant map. If $E\subseteq X$ is a $G$-invariant measurable subset of $X$ satisfying $\nu(E)>0$, then $\Theta(E)$ contains a finite $gr(\mu)$-orbit.
\end{lemma}

We now define a $G$-equivariant map $\Theta$ from $\overline{P\mathcal{O}(G,\mathcal{F})}$ to the (countable) set $D$ of finite collections of conjugacy classes of proper $(G,\mathcal{F})$-free factors. Given a tree $T\in{P\mathcal{O}(G,\mathcal{F})}$, we define $\text{Red}(T)$ to be the finite collection of proper $(G,\mathcal{F})$-free factors that occur as vertex groups of trees obtained by equivariantly collapsing some of the edges of $T$ to points. The collection $\text{Red}(T)$ is nonempty because $(G,\mathcal{F})$ is nonsporadic. Given $T\in\partial P\mathcal{O}(G,\mathcal{F})$, the set of conjugacy classes of point stabilizers in $T$ is finite \cite{Jia91}. Every point stabilizer $G_v$ is contained in a unique minimal (possibly non proper) $(G,\mathcal{F})$-free factor $\text{Fill}(G_v)$. We let $\text{Per}(T)$ be the (possibly empty) finite set of conjugacy classes of proper $(G,\mathcal{F})$-free factors that arise in this way, and we set

\begin{displaymath}
\Theta(T):=\left\{
\begin{array}{ll}
\emptyset &\text{~if~} T\in P\mathcal{AT}(G,\mathcal{F})\\
\text{Red}(T) &\text{~if~} T\in P\mathcal{O}(G,\mathcal{F})\\
\text{Dyn}(T)\cup\text{Per}(T) &\text{~if~} T\in\partial P\mathcal{O}(G,\mathcal{F})\smallsetminus P\mathcal{AT}(G,\mathcal{F})
\end{array}\right..
\end{displaymath} 

\noindent Proposition \ref{classification} implies that $\Theta(T)=\emptyset$ if and only if $T\in P\mathcal{AT}(G,\mathcal{F})$. The following lemma was proved in \cite[Lemma 3.4]{Hor14-3}. Its proof adapts to the context of $(G,\mathcal{F})$-trees.

\begin{lemma}\label{Theta-measurable}
The set $P\mathcal{AT}(G,\mathcal{F})$ is measurable, and $\Theta$ is measurable. 
\qed
\end{lemma}

\begin{prop}\label{stationary}
Let $G$ be a countable group, and $\mathcal{F}$ be a free factor system of $G$. Let $\mu$ be a probability measure on $\text{Out}(G,\mathcal{F})$, whose support generates a nonelementary subgroup of $\text{Out}(G,\mathcal{F})$. Then every $\mu$-stationary measure on $\overline{P\mathcal{O}(G,\mathcal{F})}$ is concentrated on $P\mathcal{AT}(G,\mathcal{F})$.
\end{prop}

\begin{proof}[Proof of Proposition \ref{stationary}]
Let $\nu$ be a $\mu$-stationary measure on $\overline{P\mathcal{O}(G,\mathcal{F})}$. Let $E:=\overline{P\mathcal{O}(G,\mathcal{F})}\smallsetminus P\mathcal{AT}(G,\mathcal{F})$. By Proposition \ref{classification}, the image $\Theta(E)$ does not contain the empty set. However, nonelementarity of $gr(\mu)$ implies that the only finite $gr(\mu)$-orbit in $D$ is the orbit of the empty set. Lemma \ref{disjoint-translations} thus implies that $\nu(E)=0$, or in other words $\nu$ is concentrated on $P\mathcal{AT}(G,\mathcal{F})$.
\end{proof}

\begin{cor}\label{nonempty-limit-set}
Let $H\subseteq\text{Out}(G,\mathcal{F})$ be a nonelementary subgroup of $\text{Out}(G,\mathcal{F})$. Then the $H$-orbit of any point $x_0\in P\mathcal{O}(G,\mathcal{F})$ has a limit point in $P\mathcal{AT}(G,\mathcal{F})$.
\end{cor}

\begin{proof}
Let $\mu$ be a probability measure on $\text{Out}(G,\mathcal{F})$ such that $\text{gr}(\mu)=H$. An example of such a measure is obtained by giving a positive weight $\mu(h)>0$ to every element $h\in H$, in such a way that $$\sum_{h\in H}\mu(h)=1$$ (and $\mu(g)=0$ if $g\in G\smallsetminus H$). Let $\delta_{x_0}$ be the Dirac measure at $x_0$. Since $\overline{P\mathcal{O}(G,\mathcal{F})}$ is compact \cite[Proposition 3.1]{Hor14-5}, the sequence of the Cesàro averages of the convolutions $\mu^{\ast n}\ast\delta_{x_0}$ has a weak-$\ast$ limit point $\nu$, which is a $\mu$-stationary measure on $\overline{P\mathcal{O}(G,\mathcal{F})}$, see \cite[Lemma 2.2.1]{KM96}. We have $\nu(\overline{Hx_0})=1$, where $Hx_0$ denotes the $H$-orbit of $x_0$ in $\overline{P\mathcal{O}(G,\mathcal{F})}$, and Proposition \ref{stationary} implies that $\nu(P\mathcal{AT}(G,\mathcal{F}))=1$. This implies that $\overline{Hx_0}\cap P\mathcal{AT}(G,\mathcal{F})$ is nonempty.
\end{proof}

As a consequence of Theorem \ref{boundary-fz} and Corollary \ref{nonempty-limit-set}, we get the following fact.

\begin{cor}\label{limit-set-fz}
Let $H\subseteq\text{Out}(G,\mathcal{F})$ be a nonelementary subgroup of $\text{Out}(G,\mathcal{F})$. Then the $H$-orbit of any point in $FZ(G,\mathcal{F})$ has a limit point in $\partial_{\infty} FZ(G,\mathcal{F})$.
\qed
\end{cor}

\begin{proof}[Proof of Theorem \ref{nonelementary}]
Let $\mathcal{F}$ be a free factor system of $G$, and let $H$ be a nonelementary subgroup of $\text{Out}(G,\mathcal{F})$. Corollary \ref{nonempty-limit-set} shows that the $H$-orbit of any point in $FZ(G,\mathcal{F})$ has a limit point in $\partial_{\infty} FZ(G,\mathcal{F})$. As $H$ does not fix any element in $\partial_{\infty} FZ(G,\mathcal{F})$, the conclusion follows from the classification of subgroups of isometries of Gromov hyperbolic spaces (Theorem \ref{Gromov}).  
\end{proof}

\section{The inductive argument}\label{sec-6}

\subsection{Variations over the Tits alternative}

We recall from the introduction that a group $G$ is said to satisfy the Tits alternative relative to a class $\mathcal{C}$ of groups if every subgroup of $G$ either belongs to $\mathcal{C}$, or contains a rank two free subgroup. Our main result is the following. A group $H$ is \emph{freely indecomposable} if it does not split as a free product of the form $H=A\ast B$, where both $A$ and $B$ are nontrivial.

\begin{theo}\label{Tits}
Let $\{G_1,\dots,G_k\}$ be a finite collection of freely indecomposable countable groups, not isomorphic to $\mathbb{Z}$, let $F$ be a finitely generated free group, and let $$G:=G_1\ast\dots\ast G_k\ast F.$$ Let $\mathcal{C}$ be a collection of groups that is stable under isomorphisms, contains $\mathbb{Z}$, and is stable under subgroups, extensions, and passing to finite index supergroups. Assume that for all $i\in\{1,\dots,k\}$, both $G_i$ and $\text{Out}(G_i)$ satisfy the Tits alternative relative to $\mathcal{C}$.\\ 
Then $\text{Out}(G)$ and $\text{Aut}(G)$ satisfy the Tits alternative relative to $\mathcal{C}$. 
\end{theo}

In particular, Theorem \ref{Tits} applies to the case where $\mathcal{C}$ is either the class of virtually solvable groups (see \cite[Lemme 6.11]{Can11} for stability of $\mathcal{C}$ under extensions), or the class of virtually polycyclic groups. 

Theorem \ref{Tits} will be a consequence of the following relative version. For all $i\in\{1,\dots,k\}$, let $A_i\subseteq\text{Out}(G_i)$, and let $\mathcal{A}:=(A_1,\dots,A_k)$. We recall from Section \ref{sec-relative} that $\text{Out}(G,\mathcal{F}^{A})$ denotes the subgroup of $\text{Out}(G)$ consisting of those automorphisms that preserve the conjugacy classes of all subgroups $G_i$, and induce an outer automorphism in $A_i$ in restriction to each $G_i$. 

\begin{theo}\label{Tits-2}
Let $G$ be a countable group, let $\mathcal{F}$ be a free factor system of $G$, and let $\mathcal{A}$ be as above. Let $\mathcal{C}$ be a collection of groups that is stable under isomorphisms, contains $\mathbb{Z}$, and is stable under subgroups and extensions, and passing to finite index supergroups. Assume that for all $i\in\{1,\dots,k\}$, both $G_i$ and $A_i$ satisfy the Tits alternative relative to $\mathcal{C}$.\\ Then $\text{Out}(G,\mathcal{F}^{\mathcal{A}})$ satisfies the Tits alternative relative to $\mathcal{C}$.
\end{theo}

When all subgroups in $\mathcal{A}$ are trivial, Theorem \ref{Tits-2} specifies as follows.

\begin{theo}\label{Tits-relative}
Let $G$ be a countable group, and let $\mathcal{F}$ be a free factor system of $G$. Let $\mathcal{C}$ be a collection of groups that is stable under isomorphisms, contains $\mathbb{Z}$, and is stable under subgroups and extensions, and passing to finite index supergroups. Assume that all peripheral subgroups of $G$ satisfy the Tits alternative relative to $\mathcal{C}$.\\ Then $\text{Out}(G,\mathcal{F}^{(t)})$ satisfies the Tits alternative relative to $\mathcal{C}$.
\end{theo}

In the classical case where $\mathcal{C}$ is the class of virtually solvable groups, we also mention that our proof of Theorem \ref{Tits} also provides a bound on the degree of solvability of the finite-index solvable subgroup arising in the statement.

\begin{question}
If all groups $G_i$ and $\text{Out}(G_i)$ satisfy the Tits alternative relative to the class of virtually abelian subgroups, does $\text{Out}(G)$ also satisfy the Tits alternative relative to this class ? Similarly, if all groups $G_i$ satisfy the Tits alternative relative to the class of virtually abelian subgroups, does $\text{Out}(G,\mathcal{F}^{(t)})$ also satisfy the Tits alternative relative to this class ? The issue here is that this class is not stable under extensions. Our question is motivated by the classical case of finitely generated free groups, for which Bestvina, Feighn and Handel have proved that every virtually solvable subgroup of $\text{Out}(F_N)$ is actually virtually abelian and finitely generated, with a bound on the index of the abelian subgroup that only depends on $N$ (\cite{BFH05}, see also \cite{Ali02}).
\end{question}

We first explain how to derive Theorems \ref{Tits-2} and \ref{Tits} from Theorem \ref{Tits-relative}, before proving Theorem \ref{Tits-relative} in the next section.

\begin{proof}[Proof of Theorem \ref{Tits-2}]
There is a morphism from $\text{Out}(G,\mathcal{F}^{\mathcal{A}})$ to the direct product $A_1\times\dots\times A_k$, whose kernel is equal to $\text{Out}(G,\mathcal{F}^{(t)})$. Since $\mathcal{C}$ is stable under extensions, the class of groups satisfying the Tits alternative relative to $\mathcal{C}$ is stable under extensions, so Theorem \ref{Tits-2} follows from Theorem \ref{Tits-relative}.
\end{proof}

\begin{proof}[Proof of Theorem \ref{Tits}]
Let $\mathcal{F}:=\{[G_1],\dots,[G_k]\}$. As all $G_i$'s are freely indecomposable, the group $\text{Out}(G)$ permutes the conjugacy classes in $\mathcal{F}$. Therefore, there exists a finite-index subgroup $\text{Out}^0(G)$ of $\text{Out}(G)$ which preserves all conjugacy classes in $\mathcal{F}$. For all $i\in\{1,\dots,k\}$, the group $G_i$ is equal to its own normalizer in $G$, so every element $\Phi\in\text{Out}^{0}(G)$ induces a well-defined element of $\text{Out}(G_i)$. In other words, the subgroup $\text{Out}^0(G)$ is a subgroup of $\text{Out}(G,\mathcal{F}^{\mathcal{A}})$, with $A_i=\text{Out}(G_i)$ for all $i\in\{1,\dots,k\}$. Theorem \ref{Tits} thus follows from Theorem \ref{Tits-2} (the statement for the group $\text{Aut}(G)$ also follows, because if both $G$ and $\text{Out}(G)$ satisfy the Tits alternative relative to $\mathcal{C}$, then so does $\text{Aut}(G)$).
\end{proof}

\subsection{Proof of Theorem \ref{Tits-relative}}

The proof is by induction on the pair $(\text{rk}_K(G,\mathcal{F}),\text{rk}_f(G,\mathcal{F}))$, for the lexicographic order. Let $\mathcal{F}$ be a free factor system of $G$. The conclusion holds if $\text{rk}_K(G,\mathcal{F})=1$: in this case, the group $G$ is either peripheral, or isomorphic to $\mathbb{Z}$. It also holds in the sporadic cases by Propositions \ref{sporadic-1} and \ref{sporadic-2}. We now assume that $(G,\mathcal{F})$ is nonsporadic, and let $H$ be a subgroup of $\text{Out}(G,\mathcal{F}^{(t)})$. We will show that either $H$ contains a rank two free subgroup, or $H\in\mathcal{C}$. Using Theorem \ref{trichotomy}, we can assume that either $H$ preserves a finite set of conjugacy classes of proper $(G,\mathcal{F})$-free factors, or that $H$ virtually fixes a tree with trivial arc stabilizers in $\partial P\mathcal{O}(G,\mathcal{F})$. 
\\
\\
\indent We first assume that $H$ has a finite index subgroup $H^0$ which preserves the conjugacy class of a proper $(G,\mathcal{F})$-free factor $G'$. We denote by $\text{Out}(G,\mathcal{F}^{(t)},G')$ the subgroup of $\text{Out}(G,\mathcal{F}^{(t)})$ made of those elements that preserve the conjugacy class of $G'$ (so $H^0$ is a subgroup of $\text{Out}(G,\mathcal{F}^{(t)},G')$). Since $G'$ is equal to its own normalizer in $G$, every element $\Phi\in\text{Out}(G,\mathcal{F}^{(t)},G')$ induces by restriction a well-defined outer automorphism $\Phi_{G'}$ of $G'$. The automorphism $\Phi_{G'}$ coincides with a conjugation by an element $g\in G$ in restriction to every factor in $\mathcal{F}_{G'}$ (where we recall that $\mathcal{F}_{G'}$ is the collection of $G'$-conjugacy classes of subgroups in $\mathcal{F}$ that are contained in $G'$). Since $G'$ is malnormal, we have $g\in G'$. In other words, there is a restriction morphism  $$\Psi:\text{Out}(G,\mathcal{F}^{(t)},G')\to\text{Out}(G',\mathcal{F}_{G'}^{(t)}).$$ Since $G'$ is a $(G,\mathcal{F})$-free factor, there exist $i_1<\dots<i_s$ such that $G$ splits as $$G=G'\ast G'_{i_1}\ast\dots\ast G'_{i_s}\ast F',$$ where $G'_{i_j}$ is conjugate to $G_{i_j}$ for all $j\in\{1,\dots,s\}$, and $F'$ is a finitely generated free group. We let $\mathcal{F}':=\{[G'],[G'_{i_1}],\dots,[G'_{i_s}]\}$. Then the kernel of $\Psi$ is equal to $\text{Out}(G,{\mathcal{F}'}^{(t)})$.

Recall from Equations \eqref{rkf} and \eqref{rkK} in Section \ref{sec-relative} that $\text{rk}_f(G',\mathcal{F}_{G'})+\text{rk}_f(G,\mathcal{F}')=\text{rk}_f(G,\mathcal{F})$, and $\text{rk}_K(G',\mathcal{F}_{G'})+\text{rk}_K(G,\mathcal{F}')=\text{rk}_K(G,\mathcal{F})+1$. Since $G'$ is a proper $(G,\mathcal{F})$-free factor, we either have $\text{rk}_K(G',\mathcal{F}_{G'})\ge 2$, in which case $\text{rk}_K(G,\mathcal{F}')<\text{rk}_K(G,\mathcal{F})$, or else $\text{rk}_K(G',\mathcal{F}_{G'})=\text{rk}_f(G',\mathcal{F}_{G'})=1$, in which case $\text{rk}_K(G,\mathcal{F}')=\text{rk}_K(G,\mathcal{F})$ and $\text{rk}_f(G,\mathcal{F}')<\text{rk}_f(G,\mathcal{F})$. Since $G'$ is a proper $(G,\mathcal{F})$-free factor, we also have $\text{rk}_K(G',\mathcal{F}_{G'})<\text{rk}_K(G,\mathcal{F})$. Our induction hypothesis therefore implies that both $\text{Out}(G',\mathcal{F}_{G'}^{(t)})$ and $\text{Out}(G,{\mathcal{F}'}^{(t)})$ satisfy the Tits alternative relative to $\mathcal{C}$. Since $\mathcal{C}$ is stable under extensions, the class of groups satisfying the Tits alternative relative to $\mathcal{C}$ is stable under extensions. So $H^0$, and hence $H$, satisfies the Tits alternative relative to $\mathcal{C}$. 
\\
\\
\indent We now assume that $H$ has a finite index subgroup $H^0$ which fixes the projective class of a tree $[T]\in\overline{P\mathcal{O}(G,\mathcal{F})}$ with trivial arc stabilizers. Then $H^0$ is a cyclic extension of a subgroup $H'$ that fixes a nonprojective tree $T\in\overline{\mathcal{O}(G,\mathcal{F})}$ \cite{GL14-2}. It is enough to show that $\text{Out}(T)$ satisfies the Tits alternative relative to $\mathcal{C}$.

Denote by $V$ the finite set of $G$-orbits of points with nontrivial stabilizer in $T$, and by $\{G_v\}_{v\in V}$ the collection of point stabilizers in $T$. As any element of $\text{Out}(T)$ induces a permutation of the finite set $V$, some finite index subgroup $\text{Out}^0(T)$ of $\text{Out}(T)$ preserves the conjugacy class of all groups $G_v$ with $v\in V$. As $T$ has trivial arc stabilizers, all point stabilizers in $T$ are equal to their normalizer in $G$. As above, there is a morphism from $\text{Out}^0(T)$ to the direct product of all $\text{Out}(G_v,\mathcal{F}_{G_v}^{(t)})$, whose kernel is contained in $\text{Out}(T,\{[G_v]\}^{(t)})$. 

Corollary \ref{Tits-stabilizer} shows that $\text{Out}(T,\{[G_v]\}^{(t)})$ satisfies the Tits alternative relative to $\mathcal{C}$. Since $T$ has trivial arc stabilizers, Proposition \ref{per-finite} implies that $\text{rk}_K(G_v,\mathcal{F}_{G_v})\le\text{rk}_K(G,\mathcal{F})-1$ for all $v\in V$. Therefore, our induction hypothesis implies that $\text{Out}(G_v,\mathcal{F}_{G_v}^{(t)})$ satisfies the Tits alternative relative to $\mathcal{C}$. As the Tits alternative is stable under extensions, we deduce that $\text{Out}(T)$, and hence $H$, satisfies the Tits alternative relative to $\mathcal{C}$. 
\qed

\section{Applications}\label{sec-7}

\subsection{Outer automorphisms of right-angled Artin groups}

Given a finite simplicial graph $\Gamma$, the \emph{right-angled Artin group} $A_{\Gamma}$ is the group defined by the following presentation. Generators of $A_{\Gamma}$ are the vertices of $\Gamma$, and relations are given by commutation of any two generators that are joined by an edge in $\Gamma$. As a consequence of Theorem \ref{Tits} and of work by Charney and Vogtmann \cite{CV11}, we show that the outer automorphism group of any right-angled Artin group satisfies the Tits alternative.

\begin{theo} \label{tits-raag}
For all finite simplicial graphs $\Gamma$, the group $\text{Out}(A_{\Gamma})$ satisfies the Tits alternative.
\end{theo}  

Let $N$ be the number of components of $\Gamma$ consisting of a single point, and let $\Gamma_1,\dots,\Gamma_k$ be the connected components of $\Gamma$ consisting of more than one point. Then we have $A_{\Gamma}=A_{\Gamma_1}\ast\dots\ast A_{\Gamma_k}\ast F_N$. All subgroups $A_{\Gamma_i}$ of this decomposition are freely indecomposable and not isomorphic to $\mathbb{Z}$: it is the Grushko decomposition of $A_{\Gamma}$. 

Theorem \ref{tits-raag} was first proven by Charney, Crisp and Vogtmann in the case where $\Gamma$ is connected and triangle-free \cite{CCV07}, then extended by Charney and Vogtmann in \cite{CV11} to the case of graphs satisfying some homogeneity condition, where it was noticed that the full version would follow from Theorem \ref{Tits}. We now explain how to make this deduction. The reader is referred to \cite{Cha07} for a survey paper on right-angled Artin groups, and to \cite{CCV07,CV09,CV11} for a study of their automorphism groups. 

Let $\Gamma$ be a finite simplicial connected graph. Let $v\in\Gamma$ be a vertex of $\Gamma$. The \emph{link} of $v$, denoted by $lk(v)$, is the full subgraph of $\Gamma$ spanned by all vertices adjacent to $v$. The \emph{star} of $v$, denoted by $st(v)$, is the full subgraph of $\Gamma$ spanned by $v$ and $lk(v)$. The relation $\le$ defined on the set of vertices of $\Gamma$ by setting $v\le w$ whenever $lk(v)\subseteq st(w)$ is transitive, and induces a partial ordering on the set of equivalence classes of vertices $[v]$, where $w\in [v]$ if and only if $v\le w$ and $w\le v$ \cite[Lemma 2.2]{CV09}. A vertex $v$ of $\Gamma$ is \emph{maximal} if its equivalence class is maximal for this relation. The \emph{link} $lk(\Theta)$ of a subgraph $\Theta$ of $\Gamma$ is the intersection of the links of all vertices in $\Theta$. The \emph{star} $st(\Theta)$ of $\Theta$ is the full subgraph of $\Gamma$ spanned by both $\Theta$ and its link. Given a full subgraph $\Theta$ of $\Gamma$, the group $A_{\Theta}$ embeds as a subgroup of $A_{\Gamma}$. 

Laurence \cite{Lau95}, extending work of Servatius \cite{Ser89}, gave a finite generating set of $\text{Out}(A_{\Gamma})$, consisting of graph automorphisms, inversions of a single generator, transvections $v\mapsto vw$ with $v\le w$, and partial conjugations by a generator $v$ on one component of $\Gamma\smallsetminus st(v)$.

The subgroup $\text{Out}^0(A_{\Gamma})$ of $\text{Out}(A_{\Gamma})$ generated by inversions, transvections and partial conjugations, has finite index in $\text{Out}(A_{\Gamma})$. Assume that $\Gamma$ is connected, and let $v$ be a maximal vertex. Then any element of $\text{Out}^0(A_{\Gamma})$ has a representative $f_v$ which preserves both $A_{[v]}$ and $A_{st[v]}$ \cite[Proposition 3.2]{CV09}. Restricting $f_v$ to $A_{st[v]}$ gives a \emph{restriction morphism}
\begin{displaymath}
R_v:\text{Out}^0(A_{\Gamma})\to\text{Out}^0(A_{st[v]})
\end{displaymath}

\noindent \cite[Corollary 3.3]{CV09}. The map from $A_{\Gamma}$ to $A_{\Gamma\smallsetminus [v]}$ that sends each generator in $[v]$ to the identity induces an \emph{exclusion morphism} 
\begin{displaymath}
E_v:\text{Out}^0(A_{\Gamma})\to\text{Out}^0(A_{\Gamma\smallsetminus [v]}).
\end{displaymath}

\noindent Since $v$ is a maximal vertex for the subgraph $st[v]$, and since $lk[v]=st[v]\smallsetminus [v]$, we can compose the restriction morphism on $A_{\Gamma}$ with the exclusion morphism on $A_{st[v]}$ to get a \emph{projection morphism}
\begin{displaymath}
P_v:\text{Out}^0(A_{\Gamma})\to \text{Out}^0(A_{lk[v]}).
\end{displaymath}

\noindent \cite[Corollary 3.3]{CV09}. By combining the projection morphisms for all maximal equivalence classes of vertices of $\Gamma$, we get a morphism
\begin{displaymath}
P:\text{Out}^0(A_{\Gamma})\to\prod\text{Out}^0(A_{lk[v]}),
\end{displaymath}

\noindent where the product is taken over the set of maximal equivalence classes of vertices of $\Gamma$.

\begin{prop} (Charney--Vogtmann \cite[Theorem 4.2]{CV09}) \label{amalgamated-projection}
If $\Gamma$ is a connected graph that contains at least two equivalence classes of maximal vertices, then the kernel of $P$ is a free abelian subgroup of $\text{Out}^0(A_{\Gamma})$.
\end{prop}

\begin{prop} (Charney--Vogtmann \cite[Proposition 4.4]{CV09}) \label{single-maximal}
If $\Gamma$ is a connected graph that contains a single equivalence class $[v]$ of maximal vertices, then $A_{[v]}$ is abelian, and there is a surjective morphism
\begin{displaymath}
\text{Out}(A_{\Gamma})\to GL(A_{[v]})\times\text{Out}(A_{lk([v])}),
\end{displaymath}

\noindent whose kernel is a free abelian subgroup of $\text{Out}(A_{\Gamma})$.
\end{prop}

\begin{proof} [Proof of Theorem \ref{tits-raag}]
The proof is by induction on the number of vertices of $\Gamma$. The case of a graph having a single vertex is obvious. Thanks to Theorem \ref{Tits} and the description of the Grushko decomopsition of $A_{\Gamma}$, we can assume that $\Gamma$ is connected. Let $v$ be a maximal vertex of $\Gamma$. As $lk_{[v]}$ has stricly fewer vertices than $\Gamma$, it follows from the induction hypothesis that $\text{Out}(A_{lk[v]})$ satisfies the Tits alternative, and so does $\text{Out}^0(A_{lk[v]})$. If $\Gamma$ contains a single equivalence class of maximal vertices, then it follows from Proposition \ref{single-maximal}, and from Tits' original version of the alternative for linear groups \cite{Tit72}, that $\text{Out}(A_{\Gamma})$ satisfies the Tits alternative. If $\Gamma$ contains at least two equivalence classes of maximal vertices, then it follows from Proposition \ref{amalgamated-projection} that $\text{Out}(A_{\Gamma})$ satisfies the Tits alternative. 
\end{proof}

\subsection{Outer automorphisms of relatively hyperbolic groups}

Let $G$ be a group, and $\mathcal{P}$ be a finite collection of subgroups of $G$. Following Bowditch \cite{Bow12} (see \cite{Hru10,Osi06} for equivalent definitions), we say that $G$ is \emph{hyperbolic relative to $\mathcal{P}$} if $G$ admits a simplicial action on a connected graph $\mathcal{K}$ such that

\begin{itemize}
\item the graph $\mathcal{K}$ is Gromov hyperbolic, and for all $n\in\mathbb{N}$, every edge of $\mathcal{K}$ is contained in finitely many simple circuits of length $n$, and
\item the edge stabilizers for the action of $G$ on $\mathcal{K}$ are finite, and there are finitely many orbits of edges, and
\item the set $\mathcal{P}$ is a set of representatives of the conjugacy classes of the infinite vertex stabilizers.
\end{itemize}

\begin{theo}\label{tits-rh}
Let $G$ be a torsion-free group, which is hyperbolic relative to a finite family $\mathcal{P}$ of finitely generated subgroups. Let $\mathcal{C}$ be a collection of groups that is stable under isomorphisms, contains $\mathbb{Z}$, and is stable under subgroups, extensions, and passing to finite index supergroups. Assume that for all $H\in\mathcal{P}$, both $H$ and $\text{Out}(H)$ satisfy the Tits alternative relative to $\mathcal{C}$.\\ Then $\text{Out}(G,\mathcal{P})$ satisfies the Tits alternative relative to $\mathcal{C}$.
\end{theo}

\begin{proof}
The peripheral subgroups $G_i$ arising in the Grushko decomposition of $G$ relative to $\mathcal{P}$ (see \cite{GL10} for a definition of the relative Grushko decomposition) are torsion-free, freely indecomposable relative to $\mathcal{P}_{G_i}$ (i.e. they do not split as a free product in which all subgroups in $\mathcal{P}_{G_i}$ are conjugate into one of the factors), and hyperbolic relative to $\mathcal{P}_{G_i}$. Each subgroup $G_i$ satisfies the Tits alternative relative to $\mathcal{C}$ as soon as all groups in $\mathcal{P}$ do (this follows from \cite{Gro87}). Our main result (Theorem \ref{Tits}) therefore enables us to reduce to the case where $G$ is freely indecomposable relative to $\mathcal{P}$. In this case, we can use the description of $\text{Out}(G,\mathcal{P})$ stated below, which is due to Guirardel and Levitt. Since the Tits alternative holds for mapping class groups of compact surfaces (Ivanov \cite{Iva84}, McCarthy \cite{McC85}), we deduce the Tits alternative for $\text{Out}(G,\mathcal{P})$.
\end{proof}

\begin{theo} (Guirardel--Levitt \cite[Theorem 1.4]{GL14}) \label{out-relhyp}
Let $G$ be a torsion-free group, which is hyperbolic relative to a finite family $\mathcal{P}$ of finitely generated subgroups, and freely indecomposable relative to $\mathcal{P}$. Then some finite index subgroup $\text{Out}^0(G,\mathcal{P})$ of $\text{Out}(G,\mathcal{P})$ fits in an exact sequence
\begin{displaymath}
1\to\mathcal{T}\to\text{Out}^0(G,\mathcal{P})\to\prod_{i=1}^p MCG(\Sigma_i)\times\prod_{H\in\mathcal{P}}\text{Out}(H),
\end{displaymath}

\noindent where $\mathcal{T}$ is finitely generated free abelian, and $MCG(\Sigma_i)$ is the mapping class group of a compact surface $\Sigma_i$. 
\end{theo}

When the parabolic subgroups are virtually polycyclic, we get the following result.

\begin{theo} \label{tits-relhyp}
Let $G$ be a torsion-free group, which is hyperbolic relative to a finite family of virtually polycyclic subgroups. Then $\text{Out}(G)$ satisfies the Tits alternative relative to the class of virtually polycyclic groups.
\end{theo}

\begin{proof} 
We first recall that the outer automorphism group $\text{Out}(P)$ of a virtually polycyclic group $P$ satisfies the Tits alternative relative to the class of virtually polycyclic groups. Indeed, a theorem of Auslander \cite{Aus67} establishes that $\text{Out}(P)$ embeds as a subgroup of $SL_N(\mathbb{Z})$ for some $N\in\mathbb{N}$. Tits' original statement of the Tits alternative \cite{Tit72} implies that $\text{Out}(P)$ satisfies the Tits alternative relative to the class of virtually solvable groups (every linear group over a field of characteristic $0$, finitely generated or not, satisfies the Tits alternative). In addition, a theorem of Mal'cev states that solvable subgroups of $SL_N(\mathbb{Z})$ are polycyclic \cite{Mal51}. Hence $\text{Out}(P)$ satisfies the Tits alternative relative to the class of virtually polycyclic groups. 

Denote by $\mathcal{P}$ the collection of parabolic subgroups. We can assume that $\mathcal{P}$ does not contain any virtually cyclic subgroup. Then every element of $\text{Out}(G)$ induces a permutation of the conjugacy classes of the subgroups in $\mathcal{P}$. Indeed, subgroups in $\mathcal{P}$ can be characterized as the maximal subgroups which do not contain a free subgroup of rank $2$, and are not virtually cyclic. Therefore, the group $\text{Out}(G,\mathcal{P})$ is a finite index subgroup of $\text{Out}(G)$. Theorem \ref{tits-relhyp} thus follows from Theorem \ref{tits-rh}.
\end{proof}

\bibliographystyle{amsplain}
\bibliography{/Users/Camille/Documents/Bibliographie}

\end{document}